\documentclass[11pt,reqno]{amsart}
\usepackage{amsmath,amssymb}
\theoremstyle{plain}
\newtheorem{theorem}{Theorem}
\newtheorem{prop}[theorem]{Proposition}
\newtheorem{cor}[theorem]{Corollary}
\theoremstyle{definition}
\newtheorem*{definition}{Definition}
\newtheorem*{example}{Example}
\newtheorem{lemma}[theorem]{Lemma}
\theoremstyle{remark}
\newtheorem{remark}[theorem]{Remark}
\numberwithin{theorem}{section} \numberwithin{equation}{section}
\def\c{\mathbb{C}}
\def\Z{\mathbb{Z}}
\def\R{\mathbb{R}}
\def\N{\mathbb{N}}
\def\A{\R[P]}
\def\a{\c[P]}
\def\V{V_{\R}}
\def\vreg{V_\c^{\mathrm{reg}}}
\def\v{V_{\c}}
\def\S{\mathcal{S}}
\def\psil{\psi_{\ell}}

\def\wmm{\widetilde M}
\def\wla{\widetilde\lambda}

\def\wrho{\widetilde\rho}
\def\wmu{\widetilde\mu}
\def\wdel{\widetilde\Delta}
\def\nn{\mathcal N}
\def\dst{\Delta'}
\def\rst{{R'}}
\def\alst{{\alpha'}}
\def\nnst{{\mathcal N'}}
\def\rhost{\rho'}
\def\psist{{\psi'}}
\def\wpsi{\widetilde \psi}
\def\mst{m'}
\def\fst{F'}
\begin{document}

\author{Oleg Chalykh}
\address{School of Mathematics, University of Leeds, Leeds, LS2 9JT, UK}
\email{oleg@maths.leeds.ac.uk}

\author{Pavel Etingof}
\address{Department of Mathematics, MIT, Cambridge, MA, USA}
\email{etingof@math.mit.edu}

\title{Orthogonality relations and Cherednik identities for multivariable Baker--Akhiezer functions}
\begin{abstract}
We establish orthogonality relations for the Baker--Akhiezer (BA) eigenfunctions of the
Macdonald difference operators. We also obtain a version of Cherednik--Macdonald--Mehta integral for these functions. As a corollary, we give a simple derivation of the norm identity and Cherednik--Macdonald--Mehta integral for Macdonald polynomials. In the appendix written by the first author, we prove a summation formula for BA functions. We also consider more general identities of Cherednik type, which we use to introduce and construct more general, twisted BA functions. This leads to a construction of new quantum integrable models of Macdonald--Ruijsenaars type.

\end{abstract}

\maketitle
\section{Introduction}

Around 1988, Macdonald introduced a remarkable family of multivariate orthogonal polynomials related to root systems \cite{M1}. Apart from a root system $R$, these polynomials depend on two additional (sets of) parameters $q,t$ and specialize to various families of symmetric functions, among which are the characters of simple complex Lie groups, Hall--Littlewood functions, zonal spherical functions, Jack polynomials, and multivariate Jacobi polynomials of Heckman and Opdam \cite{HO}. The Macdonald polynomials have since become a subject of numerous works revealing their links to many different areas of mathematics and mathematical physics.

The Macdonald polynomials are customarily defined as symmetric polynomial eigenfunctions of some rather remarkable partial difference operators, called Macdonald operators. These operators can be viewed as commuting quantum Hamiltonians, and the corresponding quantum  model in case $R=A_n$ is equivalent to the trigonometric limit of the Ruijsenaars model \cite{R}, a relativistic version of the Calogero--Moser model. The Macdonald polynomials play the role of eigenstates for these Macdonald--Ruijsenaars models and only exist on certain discrete energy levels. Their orthogonality follows from the fact that the Macdonald operators are self-adjoint with respect to a certain scalar product (Macdonald's product) defined as an integral over $n$-dimensional torus, with an explicit analytic measure.

For other values of the energy, the solutions to the eigenvalue problem are non-elementary functions which can be expressed in terms of $q$-Harish--Chandra series \cite{LS}. Rather remarkably, in the case $t\in q^{\Z}$ these series reduce to elementary (but still highly nontrivial) functions. These non-polynomial eigenfunctions $\psi(\lambda,x)$ depend on continuous (rather than discrete) spectral parameter $\lambda$ and can be viewed as the Bloch--Floquet (i.e.quasi-periodic) solutions to the eigenvalue problem. Such solutions were constructed and studied in \cite{C02}; in the case $R=A_n$ they were known from the earlier work \cite{FVa,ES}. As shown in \cite{C02}, the functions $\psi(\lambda,x)$ are uniquely characterized by certain analytic properties, which makes them similar to the Baker--Akhiezer functions from the finite-gap theory \cite{N,DMN,Kr1,Kr2} (see Sections \ref{secba} and \ref{fgap} below). For that reason we will refer to $\psi(\lambda,x)$ as multivariable Baker--Akhiezer (BA) functions. The idea that eigenfunctions of the quantum Calogero--Moser model for integral coupling parameters should be given by certain multivariable Baker--Akhiezer functions goes back to the work of the first author and Veselov \cite{CV}, see \cite{C08} for the survey of known results in that direction.

According to \cite{C02}, the BA functions $\psi(\lambda,x)$ are related to Macdonald polynomials by a formula that generalizes the Weyl character formula. Using this, some important properties of Macdonald polynomials were derived in \cite{C02} from analogous properties of $\psi$. In particular, the duality and evaluation identities for Macdonald polynomials are simple corollaries of the bispectral duality for $\psi$. The approach of \cite{C02} led to an elementary proof of Macdonald's conjectures, different from Cherednik's proof that uses double affine Hecke algebras \cite{Ch1, Ch2}.

Our first main result concerns a question which was not addressed in \cite{C02}, namely, the orthogonality properties of $\psi(\lambda,x)$. Since these are eigenfunctions for the Macdonald operators (which are self-adjoint with respect to Macdonald's scalar product), one would expect $\psi$ to form an orthogonal family. However, there is a subtlety here due to the fact that the definition of the Macdonald's product requires that $t=q^{m}$ with positive $m$, so it does not work for $m\in\Z_-$. In the latter case the Macdonald's product becomes degenerate and the action of Macdonald operators on symmetric polynomials becomes non-semisimple. (There is no such problem for $m\in\Z_+$, however in that case the functions $\psi(\lambda,x)$ have poles on the contour of integration, so the Macdonald's product again is not well-defined.) The way around that problem is suggested by the work of the second author and Varchenko \cite{EV}. Namely, as we show in Theorem \ref{ort1} below, the correct scalar product can be defined by shifting the contour of integration suitably, after which the integral can be easily evaluated by moving the contour to infinity. Morally, this is the same argument as the one used by Grinevich and Novikov in \cite{GN}, where they derive orthogonality relations for BA functions on Riemann surfaces. Similarly to \cite{GN}, our scalar product is indefinite. However, to compare with their situation, our $\psi(\lambda,x)$ represents a section of a line bundle not on a curve, but on a certain $n$-dimensional algebraic variety, so even the existence of $\psi$ is a non-trivial fact. Also, our situation is rather special because our BA functions are self-dual unlike those in \cite{GN}.
As an application of our result, we present a simple derivation of the norm formula for Macdonald polynomials.

Our second main result is a version of the Cherednik--Macdonald--Mehta integral identity for BA functions (Theorem \ref{mehta}). It is a generalization of the self-duality of the Gaussian $e^{-x^2}$, a basic fact about the Fourier transforms. Again, the proof is quite simple, and it easily implies the integral identity originally proved by Cherednik \cite{Ch3}, in particular, it gives a new proof for the $q$-analogue of the Macdonald--Mehta integral \cite{M3,Ch3}.

The paper finishes with an appendix written by the first author. In it we prove a version of the summation formula for $\psi$ that involves the Gaussian (Theorem \ref{jg}); this implies the result of \cite[Theorem 1.3]{Ch3}. In the final section of the appendix, we introduce twisted BA functions $\psil(\lambda,x)$, $\ell\in\Z_+$, in relation to more general integrals and sums of Cherednik--Macdonald--Mehta type. We show that the functions $\psil(\lambda,x)$ serve as common eigenfunctions for quantum integrable models of Mac\-donald--Ruijsenaars type, which we call twisted Macdonald--Ruijsenaars models. To the best of our knowledge, they are new. The commuting quantum Hamiltonians for these models look as lower order perturbations of the Macdonald operators raised to power $\ell$. Our construction of these models is implicit and is based on the construction and properties of the twisted BA functions $\psil$. It would be interesting to find an explicit construction for these twisted models using appropriately modified double affine Hecke algebras.

The structure of the paper is as follows. In Section 2 we introduce notations and recall the definitions of Macdonald scalar product, Macdonald polynomials and Macdonald operators. In Section 3 we collect definitions and main properties of the Baker--Akhiezer functions in Macdonald theory. The material is based on \cite{C02} and is not new, apart from the roots of unity and evaluation results in Sections \ref{unity} and \ref{seva}. Section 4 proves orthogonality relations for the BA functions (Theorem \ref{ort1}). In Section \ref{norms} we explain how one can use that result to compute the norms of Macdonald polynomials. We also prove orthogonality relations in the case when $q$ is a root of unity (Theorem \ref{ort2}). Section 5 establishes a version of Cherednik--Macdonald--Mehta integral for the BA functions. Following \cite{EV}, we also discuss briefly the related integral transforms and use them to rewrite the Cherednik--Macdonald--Mehta integral with the integration over a real cycle. We finish the section by re-deriving Cherednik identities for Macdonald polynomials (Theorem \ref{mm}) and discuss some special cases, including $q$-Macdonald--Mehta integral. The paper concludes with an Appendix consisting of two sections. Section 6 is devoted to the proof of the summation formula analogous to \cite[Theorem 1.3]{Ch3}. In Section 7 we introduce twisted BA functions, prove their existence, and show that they serve as eigenfunctions of a twisted version of Macdonald--Ruijsenaars models.

Part of the motivation behind this work was to find analogues of Cherednik's results for the deformed root systems, discovered in \cite{CFV98,CFV99}. Since our proofs do not require double affine Hecke algebras, they can be adapted for the deformed cases. This will be a subject of a separate publication \cite{C11}.

Let us finish by mentioning that in the case $R=A_n$ the results of Theorems \ref{ort1} and \ref{mehta} were obtained previously in \cite{EV} by using representation theory of quantum groups. The strategy of \cite{EV} was in a sense opposite to the one employed in the present paper. Namely, the results in \cite{EV} were first derived in the symmetric setting, by representation-theoretic methods from \cite{EK1,EK2,ES,EV1}, and then they were extended to statements about $\psi$ by analytic arguments. In contrast, we prove our results directly for $\psi$, and then use them to derive analogous results for Macdonald polynomials. In both approaches, Proposition \ref{res} below plays the crucial role.

\medskip

\noindent {\it Acknowledgments.} The preliminary version of these
results were presented by the first author (O.~C.) at the Banff
workshop 'New developments in univariate and multivariate orthogonal
polynomials' in October 2010. O.~C. thanks the organizers for their
kind invitation. The work of P.~E. was partially supported by the
NSF grant DMS-1000113.  We would like to thank the referee
for carefully reading the manuscript and suggesting various
improvements to the exposition.

\section{Macdonald polynomials and Macdonald operators}\label{notations}

\subsection{Notations}\label{nota}
Let $\V$ be a finite-dimensional real Euclidean vector space with the scalar product $\langle\cdot\,,\cdot\rangle$.
Let $R=\{\alpha\}\subset \V$ be a reduced irreducible root system and $W$ be the Weyl group of $R$, generated by orthogonal reflections $s_{\alpha}$ for $\alpha\in R$. The dual system is
$R^\vee=\{\alpha^\vee=\frac{2\alpha}{\langle\alpha,\alpha\rangle}\,|\,\alpha\in R\}$. We choose a basis of simple roots
$\{\alpha_1,\dots,\alpha_n\}\subset R$ and denote by $R_+$ the positive half with respect to that choice, i.e.
$R_+=R\cap C_+$, where $C_+$ is the cone generated over $\R_{\ge 0}$ by the simple roots $\alpha_1, \dots, \alpha_n$. We will use the standard notation of \cite{B}, so $Q=Q(R)$ and $P=P(R)$ denote the root and weight lattices of $R$, with $Q^\vee:=Q(R^\vee)$, $P^\vee:=P(R^\vee)$. Let $Q_+=Q\cap C_+$ and $P_+=P\cap C_+$ denote the positive cones of the root and weight lattices, respectively. We reserve the notation $P_{++}$ for the dominant weights:
\begin{equation*}
P_{++}=\{\pi\in P\,|\, \langle \pi, \alpha_i\rangle\ge 0\quad\forall\ i\}\,.
\end{equation*}

Let $\A$ be the group algebra of the weight lattice $P$. We choose $0<q<1$ and think of the elements in $\A$ as functions on $\V$ of the form
\begin{equation*}
f(x)=\sum_{\nu\in P}f_{\nu}q^{\langle\nu,x\rangle}\quad\text{with}\ \,f_\nu\in\R\,.
\end{equation*}
We can view such $f$ as an analytic function on the complexified space $\v=\V\oplus i\V$ by defining $q^{\langle\nu,x\rangle}:=e^{\log q {\langle\nu,x\rangle}}$. We can also allow complex coefficients and view $f\in\a$
in a similar way. Clearly, such $f$ are periodic with the lattice of periods $\kappa Q^\vee$, where
\begin{equation}\label{kappa}
\kappa=\frac{2\pi i}{\log q}\,.
\end{equation}
Note that $\kappa\in i\R_-$. Later we will allow complex $q\ne 0$; in that case one needs to fix a value of $\log q$ so $\kappa$ might no longer be purely imaginary. Whenever we allow $q$ to vary, we do it by choosing a local branch of $\log q$.

There are three types of Macdonald's theory; they correspond to \cite{M4}, (1.4.1)--(1.4.3). The first two types
are associated to any reduced root system $R$ and one or two additional parameters. The third type corresponds to the non-reduced affine root system $(C_n^\vee,C_n)$; this case involves $5$ parameters and is related to Koornwinder polynomials \cite{Ko}. Following \cite{LS}, we will refer to these as cases {\bf a}, {\bf b} and {\bf c}, respectively. Each case depends on a data $(R,m)$ consisting of a root system $R$ and certain labels $m$ playing a role of parameters.

\subsubsection{Cases {\bf a}, {\bf b}}
Given an arbitrary reduced irreducible root system $R$, let us choose $W$-invariant multiplicity labels $m_\alpha\in \R$ for all $\alpha\in R$. These labels must be the equal for the roots of the same length, so $m_\alpha$ take at most two values, depending on whether $R$ consists of one or two $W$-orbits.

Let us introduce quantities $q_\alpha$ for $\alpha\in R$ as follows:
\begin{equation}\label{qab}
q_\alpha=\begin{cases}\,q\qquad\ \ \text{in case {\bf a}}\,,\\\,q^{\frac{\langle\alpha,\alpha\rangle}{2}}\quad\text{in case {\bf b}}\,.\end{cases}
\end{equation}
(By default, we also assume that $q_\alpha=q$ in case {\bf c}.)  We will also write $t_\alpha$ for $t_\alpha = q_\alpha^{-m_\alpha}$.

\subsubsection{Case {\bf c}}
Consider $\V=\R^n$ with the standard Euclidean product and let $R\subset \V$ be the root system of type $C_n$, that is
$R=2R^1\cup R^2$ where
\begin{equation}\label{r12}
R^1=\{\pm e_i\,|\,i=1,\dots, n\}\,,\quad R^2= \{\pm e_i\pm e_j\,|\, 1\le i<j\le n \}\,.
\end{equation}
Choose real parameters $m_i$, $i=1,\dots, 5$ and set
\begin{equation}\label{cmalpha}
m_\alpha=\begin{cases}\,\frac12+\frac12\sum_{i=1}^4 m_i\quad\text{for}\ \alpha\in 2R^1\,,\\\,m_5\quad\text{for}\ \alpha\in R^2\,.\end{cases}
\end{equation}
Below we will need the dual parameters $\mst_i$ defined as follows:
\begin{eqnarray}
 \nonumber \mst_1 &=& \frac12 +\frac12 (m_1+m_2+m_3+m_4)\,,\\
 \nonumber \mst_2 &=& -\frac12 +\frac12 (m_1+m_2-m_3-m_4)\,,\\
 \label{dualc} \mst_3 &=& -\frac12 +\frac12 (m_1-m_2+m_3-m_4)\,,\\
 \nonumber \mst_4 &=& -\frac12 +\frac12 (m_1-m_2-m_3+m_4)\,,\\
 \nonumber \mst_5 &=& m_5\,.
\end{eqnarray}
Write $t$ for
\begin{equation}\label{tc}
(t_1,t_2,t_3,t_4,t_5):=(q^{-m_1}, q^{-m_2}, -q^{-m_3}, -q^{-m_4}, q^{-m_5})\,.
\end{equation}

\medskip

In all three cases $m$ will denote the set of $m_\alpha$ or $m_i$, respectively, and we will use the abbreviation $t=q^{-m}$ to denote the above $t_\alpha$ or $t_i$.

It will be convenient to use the notation $\alst$ as follows: $\alst=\alpha^\vee$ in case {\bf b}, while $\alst=\alpha$ for cases {\bf a} and {\bf c}. Let $\rst=\{\alst\,|\,\alpha\in R\}$, that is, $\rst=R^\vee$ in case {\bf b} and $\rst=R$ in the remaining cases. To have uniform notation, let us also introduce $\mst_\alpha=m_\alpha$ {in cases {\bf a}, {\bf b}}, while in case {\bf c} we put, according to \eqref{dualc}, \eqref{cmalpha},
\begin{equation}\label{cm*}
\mst_\alpha=\begin{cases}m_1\quad\text{for}\ \alpha\in R^1\,,\\\, m_\alpha\quad\text{for}\ \alpha\in R^2\,.
\end{cases}
\end{equation}

Let us now introduce the {\bf Macdonald weight function} $\nabla$.
In cases {\bf a}, {\bf b} it is defined as follows (\cite[(5.1.28)]{M4}):
\begin{equation}\label{Delta}
\nabla=\nabla(x; q, t)=\prod_{\alpha\in
R}\frac{\left(q^{\langle\alpha,x\rangle};q_\alpha\right)_\infty}
{\left(t_\alpha q^{\langle\alpha,x\rangle};q_\alpha\right)_\infty}\,,
\end{equation}
where we used the standard notation $$(a;q)_\infty:=\prod_{i=0}^\infty (1-aq^i)\,.$$
In case {\bf c} we put (\cite[(5.1.28)]{M4})
\begin{equation}\label{cnabla}
\nabla=\nabla(x;q,t)=\nabla^{(1)}\nabla^{(2)}\,,
\end{equation}
where
\begin{gather*}
\nabla^{(1)}=\prod_{\alpha\in R^1}\frac{\left(q^{2\langle\alpha,x\rangle};q\right)_\infty }
{\prod_{i=1}^4\left(t_iq^{\langle\alpha,x\rangle};q\right)_\infty}
\intertext{and}
\nabla^{(2)}=\prod_{\alpha\in R^2}
\frac{\left(q^{\langle\alpha,x\rangle};q\right)_\infty}
{\left(t_5q^{\langle\alpha,x\rangle};q\right)_\infty}\,.
\end{gather*}

Finally, let $\rho, \rhost$ be the following vectors:
\begin{equation}\label{rho}
\rho=\frac12\sum_{\alpha\in R_+} m_\alpha\alpha\,,\quad\rhost=\frac12\sum_{\alpha\in R_+} \mst_\alpha\alst\,.
\end{equation}

\begin{remark}\label{type1}
Our notation slightly differs from the one used in \cite{M4}. First, in case {\bf b} our $R$ corresponds to $R^\vee$ in \cite{M4}. Also, Macdonald uses the parameters
\begin{equation}\label{kab}
k_\alpha=-m_\alpha\,.
\end{equation}
In case {\bf c} the relation between $k_i$ used in \cite[Section 5]{M4} and our $m_i$ is as follows:
\begin{equation}\label{kc}
m_1=-k_1\,, m_2=-k_3-\frac12\,, m_3=-k_2\,, m_4=-k_4-\frac12\,, m_5=-k_5\,.
\end{equation}
Let us also remark on the notation used in \cite{C02}. Note that in case {\bf c} we have chosen $R=C_n$, and not $B_n$ as in \cite{C02}; this is done for purely notational reasons and agrees with \cite{M4}. In case {\bf c} the above $t_i$ correspond to $a,b,-c,-d,t$ in \cite[Section 6]{C02}, while $m_i$ relate to $(k,l,l',m,m')$ used in \cite{C02} by
\begin{equation*}
(m_1,m_2,m_3,m_4,m_5)=(l,l',m,m',k)\,.
\end{equation*}
More importantly, \cite{C02} uses $q^2$ everywhere in place of the present $q$.
\end{remark}

\subsubsection{Integrality assumptions}\label{para}
Below we will mostly deal with the case when the parameters $m$ are (half-)integers, so let us introduce some additional notation for that case. Our running assumption will be that
\begin{equation}\label{abint}
m_\alpha\in\Z_+=\{0,1,2,\dots\}\quad\forall\ \alpha\in R\quad(\text{cases {\bf a} and {\bf b}})
\end{equation}
and
\begin{gather}\label{cint}
m_1\pm m_2\in\frac12+\Z\,,\quad m_3\pm m_4\in\frac12+\Z\,,\\\label{cint1}
m_i, m_i'\ge -1/2\quad\text{for}\ i=1,\dots,4\,,\quad m_5\in\Z_+\,.
\end{gather}
The first assumption means that each pair $(m_1,m_2)$ and $(m_3,m_4)$ consists of an integer and a half-integer.
For brevity, we will refer to $m$ satisfying \eqref{abint}--\eqref{cint1} as {\it integral} parameters.

The following notation will be used below for $a,b,c\in \R$:
\begin{equation}\label{prec}
a\preccurlyeq (b,c)\quad\Leftrightarrow\quad a\in \{b-\Z_+\}\cup \{c-\Z_+\}\,.
\end{equation}
For example, $0<s\preccurlyeq (3/2, 2)$ means that $s\in\{1/2,3/2,1,2\}$, while $0<s\preccurlyeq (-1/2, 2)$ means that $s\in\{1,2\}$.

\subsubsection{Weight function for $t=q^{-m}$}\label{del}

Let us write explicitly the Macdonald weight function $\nabla$ for integral parameters $m$ as specified above. It will be convenient to introduce another function $\Delta$ as follows. In case {\bf a} and {\bf b} we put
\begin{equation}\label{Deltam}
\Delta(x)=\prod_{\alpha\in R_+}\prod_{j=1}^{m_\alpha}\left(q_\alpha^{-j/2}q^{\langle\alpha,x\rangle/2}-
q_\alpha^{j/2}q^{-\langle\alpha,x\rangle/2}\right)\,.
\end{equation}
In case {\bf c}, we put $\Delta=\Delta_+^{(1)}\Delta_-^{(1)}\Delta^{(2)}$ where
\begin{gather}\label{Delta1m}
\Delta_+^{(1)}(x)=\prod_{\alpha\in R_+^1}\,\prod_{0<s\preccurlyeq(m_1,m_2)}\left(q^{-s/2}q^{\langle\alpha,x\rangle/2}-
q^{s/2}q^{-\langle\alpha,x\rangle/2}\right)\,,\\\label{Delta1p}
\Delta_-^{(1)}(x)=\prod_{\alpha\in R_+^1}\,\prod_{0<s\preccurlyeq(m_3,m_4)}\left(q^{-s/2}q^{\langle\alpha,x\rangle/2}+
q^{s/2}q^{-\langle\alpha,x\rangle/2}\right)\,,
\intertext{and}
\label{Delta2}
\Delta^{(2)}(x)=\prod_{\alpha\in R_+^2}\prod_{j=1}^{m_\alpha}\left(q_\alpha^{-j/2}q^{\langle\alpha,x\rangle/2}-q_\alpha^{j/2}q^{-\langle\alpha,x\rangle/2}\right)\,.
\end{gather}

This is related to $\nabla$ \eqref{Delta}, \eqref{cnabla} by
\begin{equation}\label{delnab}
\nabla(x;q,q^{-m})={C}\left({\Delta(x)\Delta(-x)}\right)^{-1}\,,
\end{equation}
where
\begin{equation}\label{cc}
C=\prod_{\alpha\in R_+}q_\alpha^{m_\alpha(m_\alpha+1)/2}\qquad(\text{cases {\bf a}, {\bf b}})\,,
\end{equation}
or
\begin{equation}\label{ccc}
C=\prod_{\genfrac{}{}{0pt}{}{0<r\preccurlyeq(m_1,m_2)}{0<s\preccurlyeq(m_3,m_4)}}q^{n(s+r)}\prod_{\alpha\in R_+^2}q^{m_\alpha(m_\alpha+1)/2}\qquad(\text{case {\bf c}})\,.
\end{equation}

Finally, if $\Delta=\Delta_{R,m}$ is as above then $\dst$ will denote the dual function $\dst=\Delta_{\rst, \mst}$.

\subsection{Macdonald scalar product}

Let $\nabla(x;q,t)$ be the Macdonald weight function \eqref{Delta}--\eqref{cnabla} associated to $(R,m)$. We are going to define a scalar product on $\A$, where $P=P(R)$ is the weight lattice of $R$. Let us first assume that the parameters $m$ are of the form \eqref{kab} or \eqref{kc}, respectively, with all $k_\alpha$ or $k_i$ positive integers. In that case it is easy to check that $\nabla\in\A$. For instance, in cases {\bf a} and {\bf b},
\begin{equation}\label{Deltak}
\nabla=\prod_{\alpha\in
R}\prod_{i=0}^{k_\alpha-1}(1-q_\alpha^{i}q^{\langle\alpha,x\rangle})\,.
\end{equation}
Then the Macdonald scalar product on $\A$ is defined by
\begin{equation}\label{scpr}
\langle f , g \rangle={\tt CT}\left[f(x)g(-x)\nabla(x)\right]\quad \forall f,g\in \A\,,
\end{equation}
where ${\tt CT}$ is the linear functional on $\A$ computing the constant term:
\begin{equation*}
    {\tt CT}\left[q^{\langle\nu,x\rangle}\right]=\delta_{\nu, 0}\,.
\end{equation*}

We can rewrite $\langle f , g \rangle$ as an integral over a torus. Namely, if $\kappa$ is as in \eqref{kappa} then
\begin{equation*}
    \int_{i\V/\kappa Q^\vee}q^{\langle\nu,x\rangle}\,dx=\delta_{0,\nu}\quad\text{and}\quad \int_{i\V/\kappa Q^\vee} fd\,x={\tt CT}[f]\quad\forall f\in \A\,,
\end{equation*}
where $dx$ is the normalized Haar measure on the torus $T=i\V/\kappa Q^\vee$. The scalar product \eqref{scpr} can therefore be written as
\begin{equation}\label{scint}
\langle f , g \rangle=\int_{i\V/\kappa Q^\vee} f(x)g(-x)\nabla(x)\, dx\,.
\end{equation}
Note that $\nabla(x)$ is real on $i\V$, and also for any $f\in\A$ we have that $f(-x)=\overline{f(x)}$. This implies that the scalar product \eqref{scpr} is positive definite.

For other values of the parameters, the usual convention is to define $\langle f, g \rangle$ by
analytic continuation in $t$ from the above values $t=q^k$. It is easy to see that the restriction of $\nabla$ on $i\V\subset \v$ depends analytically on $t$ provided that $t_\alpha$ (or $t_i$ in case {\bf c}) belong to $(0,1)$. Therefore, for such parameters the scalar product is still given by the integral \eqref{scint}.
However, for other values of parameters the integral \eqref{scpr} no longer gives the correct scalar product. Indeed, in the process of analytic continuation one might need to deform the contour of integration when the poles of the weight function $\nabla$ cross through $i\V$. It is far from obvious how to define the correct scalar product by an analytic formula similar to \eqref{scint} so that it would remain valid for all $t$.
The present paper provides a (partial) solution to that problem in the case $t=q^{-m}$ with integral $m$. As we will see below, a simple recipe in that case is to shift the integration cycle $i\V$ by a suitable $\xi\in \V$ (we borrowed that idea from \cite{EV}). Note that on the shifted cycle $f(-x)$ is no longer equal to the complex conjugate of $f(x)$, therefore we cannot expect the scalar product to remain positive. This has obvious parallels with the work \cite{GN}, where indefinite scalar products were associated with the Baker-Akhiezer functions on Riemann surfaces. This is not surprising, since the Baker--Akhiezer functions considered in the present paper can be viewed as multivariable analogues of some of the Baker-Akhiezer functions appearing in the finite-gap theory \cite{N,DMN,Kr1,Kr2}.

\subsection{Macdonald polynomials}

We write $\A^W$ for the $W$-invariant part of $\A$. As a vector space, $\A^W$
is generated by the orbitsums
\begin{equation}
\label{orb} \mathfrak{m}_\lambda = \sum_{\tau\in W\lambda}
q^{\langle\tau,x\rangle}\,,\qquad \lambda\in P_{++}\,.
\end{equation}

\begin{definition} Define polynomials $p_\lambda=p_\lambda(x; q,t)$ as the (unique) elements of $\A^W$ of the form
\begin{equation}\label{macp}
p_\lambda= \mathfrak{m}_\lambda+\sum_{\nu < \lambda}a_{\lambda\nu} \mathfrak{m}_\nu\,,\qquad \lambda\in P_{++}\,,
\end{equation}
which are orthogonal with respect to the scalar product \eqref{scint}:
\begin{equation}\label{mort}
\langle p_\lambda\,,\, p_\mu \rangle = 0\quad\text{for}\ \lambda\ne \mu\,.
\end{equation}
Here $a_{\lambda\nu}$ depend on $q,t$ and $\nu<\lambda$ denotes that $\nu,\lambda\in P_{++}$ with $\lambda-\nu\in P_+\setminus \{0\}$.
\end{definition}

The polynomials $p_\lambda$ were introduced by Macdonald in \cite{M1} in cases {\bf a} and {\bf b} (and some subcases of {\bf c}). In case {\bf c} they are due to Koornwinder \cite{Ko}. We will call $p_\lambda$ Macdonald polynomials in all three cases. The existence of such $p_\lambda$ is a non-trivial fact. Originally, $p_\lambda$ were constructed in \cite{M1,Ko} as eigenfunctions of the form \eqref{macp} for certain remarkable difference operators $D^\pi$, discussed in the next section. Later, Cherednik developed his celebrated DAHA theory which, among many other things, led to an alternative construction of $D^\pi$ and $p_\lambda$ \cite{Ch1,Ch2,Ch3}. Cherednik's approach was extended to Koornwinder polynomials in \cite{No,S,St}.

\begin{remark}
In cases {\bf a} and {\bf b} the above definition is usually given for the standard dominance ordering on $P_{++}$, i.e. with $\nu\le \mu$ meaning that $\mu-\nu\in Q_+$. It is easy to see (using the uniqueness of $p_\lambda$) that replacing the dominance ordering with any weaker partial ordering leads to the same polynomials. This allows for a uniform notation for all three cases.
\end{remark}

\begin{remark}\label{semi}
One should keep in mind that the coefficients $a_{\lambda\nu}$ in \eqref{macp} are certain rational functions of  $q_\alpha$ and $t_\alpha$. They may have poles and, as a result, some of $p_\lambda$ do not exist for certain values of $q,t$. This happens, for instance, in the case when $t=q^{-m}$ with integral $m$, and it is this case which will be of our main interest below.
\end{remark}

\subsection{Macdonald difference operators}\label{mops}

For any $\tau \in \v$, $T^\tau$ will denote the shift operator, which acts on a function of $x\in\v$ by $\left(T^\tau f\right)(x)=f(x+\tau)$. A difference operator $D$ (on a lattice $L\subset \v$) is a finite sum of $a_\tau(x)T^\tau$ with $\tau\in L$. The Macdonald operators (and Koornwinder operator in case {\bf c}) are certain remarkable difference operators whose eigenfunctions are the polynomials $p_\lambda$. Each $D=D^{\pi}$
is of the form
\begin{equation}\label{mop}
D^\pi=\sum_{\tau\in W\pi}a_\tau(x)(T^\tau-1)+a_0\,,
\end{equation}
for certain very specific $\pi\in V$ and with some explicitly given $a_\tau(x)$ and constant $a_0$. They were introduced in \cite{M1} for cases {\bf a}, {\bf b} and in \cite{Ko} in case {\bf c} (in case $R=A_{n}$ they also appeared in \cite{R}, see Example below).

In cases {\bf a} and {\bf b} the Macdonald operators are labeled by the minuscule and quasi-minuscule elements
$\pi\in P(R')$. Recall that a nonzero weight $\pi \in P(R)$ is called minuscule if
$\langle\pi,\alpha^\vee\rangle\in\{-1, 0, 1\}$ for all $\alpha \in R$. It is known that minuscule dominant weights are in one-to-one correspondence with nonzero elements of $P/Q$, which means that they do not exist for $R=E_8, F_4, G_2$, see \cite{B}. A weaker notion is that of a quasi-minuscule weight. By definition, $\pi \in P(R)$ is called quasi-minuscule if $\pi\in R$ and $\langle\pi,\alpha^\vee\rangle\in\{-1, 0, 1\}$
for all $\alpha \in R \backslash \{\pm\pi\}$. (Note that for for $\alpha=\pm \pi$ we have $\langle\pi, \alpha^\vee\rangle=\pm 2$.) Quasi-minuscule weights exist for all $R$ and are of the form $\pi=w\theta$, $w\in W$, where $\theta^\vee$ is the maximal coroot in $R^\vee_+$.

Given a (quasi-)minuscule $\pi\in P(R')$, the Macdonald difference operator $D^\pi$ of type {\bf a}--{\bf b} has the form \eqref{mop} with
\begin{equation}
\label{m2} a_\tau = \prod_{\genfrac{}{}{0pt}{}{\alpha \in R
:}{\langle\alpha,\tau\rangle >0}} \frac {1-t_\alpha
q^{\langle\alpha,x\rangle}}
{t_\alpha^{1/2}(1- q^{\langle\alpha,x\rangle})} \prod_{\genfrac{}{}{0pt}{}{\alpha
\in R :}{\langle{(\alpha')^\vee, \tau\rangle = 2}}} \frac {1-t_\alpha
q_\alpha q^{\langle\alpha,x\rangle}}
{t_\alpha^{1/2}(1-q_\alpha q^{\langle\alpha,x\rangle})}\,,
\end{equation}
where $t_\alpha=q_\alpha^{-m_\alpha}$ as before and the constant $a_0$ is given by
\begin{equation}
\label{m3}
a_0= \mathfrak m_\pi(-\rho)=\sum_{\tau\in W\pi} q^{-\langle\rho,\tau\rangle}\,.
\end{equation}

\begin{remark}\label{cover}
When $\pi$ is minuscule, the second product in \eqref{m2} is trivial. In that case the expression $a_0-\sum_{\tau\in W\pi}a_\tau$ cancels out and the formula for $D^\pi$ reduces to
\begin{equation}
\label{m} D^\pi = \sum_{\tau\in W\pi}a_\tau T^{\tau}, \quad a_\tau = \prod_{\genfrac{}{}{0pt}{}{\alpha \in R :}{\langle\alpha,
\tau\rangle>0}}\, \frac{1-t_\alpha
q^{\langle\alpha,x\rangle}}
{t_\alpha^{1/2}(1- q^{\langle\alpha,x\rangle})}\,.
\end{equation}
\end{remark}

In case {\bf c} we have $R=2R^1\cup R^2$ of type $C_n$. In this case we take $\pi=e_1$, so $W\pi=R^1=\{\pm e_i\,|\, i=1,2,\dots, n\}$. The corresponding operator $D^\pi$ is called the Koornwinder operator and it is given by \eqref{mop} with
\begin{equation}
\label{m4} a_\tau(x) = v(\langle\tau,x\rangle)\prod_{\genfrac{}{}{0pt}{}{\alpha \in R^2
:}{\langle\alpha,\tau\rangle >0}} \frac {1-t_\alpha
q^{\langle\alpha,x\rangle}}
{t_\alpha^{1/2}(1- q^{\langle\alpha,x\rangle})}\,,
\end{equation}
where $v$ is the following function of one variable:
\begin{equation}
\label{m5}
v(z)=\left(\frac{q}{t_1t_2t_3t_4}\right)^{1/2}\frac{\prod_{i=1}^4 (1-t_iq^{z})}{(1-q^{2z})(1-q^{2z+1})}\,.
\end{equation}
The constant $a_0$ is given by the same formula \eqref{m3}.

\begin{example}
In case $R = A_{n-1} = \{ \pm\left(e_i-e_j\right) \vert i<j \}
\subset {\mathbb R}^{n}$ with $m_\alpha \equiv m$, each fundamental
weight $\pi_s = e_1+\dots +e_s$ ($s=1,\dots, n$) is minuscule
and the corresponding operators $D_s=D^{\pi_s}$ have the form
\begin{equation}
\label{ru} D_s=\sum_{\genfrac{}{}{0pt}{}{I\subset
\{1,\dots,n\}}{|I|=s}} \prod_{\genfrac{}{}{0pt}{}{i\in I} {j\notin
I}} \frac {q^{-m+x_i-x_j}-q^{m-x_i+x_j}}
{q^{x_i-x_j}-q^{-x_i+x_j}}\, T^{I}\,,
\end{equation}
where $T^I$ stands for $\prod_{i\in I}T^{e_i}$. The (commuting) operators $D_1,\dots, D_n$ are known as the conserved quantities (or `higher Hamiltonians ') of the quantum trigonometric {\it Ruijsenaars model} \cite{R}, see also \cite{RS} for its classical counterpart.
\end{example}

\begin{remark}\label{cin} It follows from Cherednik's work that there exist in general $n=\mathrm{rank}(R)$ independent commuting difference operators that include the operators $D^\pi$ among them. We will refer to them as Cherednik--Macdonald operators. Apart from the $R=A_n$ case as in the above example, the complete set of such operators is known explicitly in some other cases (including the general case {\bf c}), see \cite{vD,vDE}.
\end{remark}

The polynomials $p_\lambda$ can be uniquely characterized as symmetric eigenfunctions of the difference operators $D^\pi$.

\begin{theorem}\label{ueig}[\cite{M1,Ko}]
In each of the cases {\bf a}--{\bf c} and for generic (or indeterminate) parameters $t=\{t_\alpha\}$, the polynomials $p_\lambda$ can be uniquely characterised as eigenfunctions of the operators $D^\pi$ of the form \eqref{macp}.
\end{theorem}
Let us sketch the proof. The first observation is that the action of each of $D^\pi$ on $\A^W$ is lower-triangular in the following sense:
\begin{equation}\label{ltr}
D^\pi(\mathfrak m_\lambda)=c_{\lambda\lambda}\mathfrak m_\lambda+\sum_{\mu<\lambda}c_{\lambda\mu}\mathfrak m_\mu\,,\qquad c_{\lambda\lambda}=\mathfrak m_{\pi}(\lambda-\rho)\,.
\end{equation}
Furthermore, one can show that for generic $t$ (and suitably chosen $\pi$) the diagonal coefficients $c_{\lambda\lambda}$ with $\lambda\in P_{++}$ are pairwise distinct:
\begin{equation}\label{ssimple}
\mathfrak m_{\pi}(\lambda-\rho)\ne \mathfrak m_{\pi}(\mu-\rho)\quad\text{for}\ \lambda\ne\mu\,.
\end{equation}
Therefore, in such case the action of $D^\pi$ on $\A^W$ is diagonalisable, so it has uniquely defined eigenfunctions  of the form \eqref{macp}. Finally, one checks that the operator $D^\pi: \A^W\to \A^W$ is self-adjoint with respect to the Macdonald scalar product. Therefore, the above eigenfunctions will be pairwise orthogonal and, therefore, coincide with the Macdonald polynomials. 

\begin{remark}\label{ssu}
We should warn the reader that the above proof does not work for some singular values of the parameters $t=q^{-m}$ for which the action of $D^\pi$ on $\A^W$ fails to be semisimple. In particular, this happens when $m_\alpha$ are positive integers (cf. Remark \ref{semi} above). In this situation, some of the Macdonald polynomials are not well-defined. However, for any given $t=t_0$, whenever a particular eigenvalue $c_{\lambda\lambda}$ of $D^\pi$ is simple for $t=t_0$, the corresponding eigenfunction will be locally analytic in $t$. Therefore, this eigenfunction must coincide with the Macdonald polynomial $p_\lambda$ specialised at $t=t_0$.
\end{remark}

\begin{remark}\label{ssuu}
The Macdonald polynomials $p_\lambda$ are known to be eigenfunctions for the full family of the Cherednik--Macdonald operators, see Remark \ref{cin}. Therefore, if we know a particular $W$-invariant eigenfunction for the whole family, and if the corresponding eigenvalue is simple for at least one operator in this family, then the corresponding eigenfunction coincides with $p_\lambda$. This implies that $p_\lambda$ is well-defined if
$W(\lambda-\rho)\ne W(\mu-\rho)$ for all $\mu\in P_{++}\setminus\{\lambda\}$.
(In fact, this only needs to be checked for $\mu<\lambda$, due to the lower-triangular nature of the Cherednik--Macdonald operators, cf. \eqref{ltr}.)
\end{remark}

\section{Baker--Akhiezer function for Macdonald operators}
\label{baf}

Throughout this section we assume $q\in\c^\times$ is not a root of unity (unless specified otherwise).
From now on we will work under the integrality assumptions as specified in \ref{para}.
For a given $(R,m)$, the Baker--Akhiezer functions (BA functions for short) are eigenfunctions of special form for the Macdonald operators $D^\pi$ with $t=q^{-m}$ (or Koornwinder operator in case {\bf c}). In cases {\bf a} and {\bf c} they were introduced and studied in \cite{C02}; case {\bf b} is entirely similar.

Let us denote by $\nn\subset \V$ the following polytope associated to $(R,m)$:
\begin{equation}
\label{nu} \nn = \{ \frac12\sum_{\alpha \in R_+}
l_\alpha \alpha \mid -m_\alpha\le l_\alpha \le m_\alpha \}\,.
\end{equation}
By $\nnst$ we denote the counterpart of $\nn$ for $(\rst, \mst)$, i.e.
 \begin{equation}
\label{nust} \nnst = \{ \frac12\sum_{\alpha \in R_+}
l_\alpha \alst \mid -\mst_\alpha\le l_\alpha \le \mst_\alpha \}\,.
\end{equation}
Note that the vertices of $\nn$ and $\nnst$ are of the form $w\rho$ and $w\rhost$, respectively, with $w\in W$.
Below $P$ and $P'$ stand for the weight lattices of $R$ and $\rst$, respectively.

\subsection{Baker--Akhiezer function}\label{secba}

Let $\psi(\lambda, x)$ be a function of $(\lambda, x)\in \v\times\v$ of the
form
\begin{equation}\label{psi}
\psi = q^{\langle\lambda, x\rangle}\sum_{\nu\in\mathcal N\cap\, \rho+P} \psi_\nu(\lambda) q^{\langle\nu,x\rangle}\,.
\end{equation}
We will assume that $\psi$ is meromorphic in $\lambda$. Let us assume that $\psi$ has the following properties for
each $\alpha\in R$ in cases {\bf a}, {\bf b} or $\alpha\in R^2$ in case {\bf c}, and for every $j=1,\dots ,m_\alpha$:
\begin{equation}\label{axpsi}
\psi\left(\lambda,\, x+{\frac 12 j\alst}\right)=\psi\left(\lambda,
\,x{-\frac 12 j\alst}\right)\quad\text{when}\ \
q^{\langle\alpha,x\rangle}=1\,.
\end{equation}
(The equality in \eqref{axpsi} is understood as an equality of meromorphic functions of $\lambda$.)

In case {\bf c}, we require in addition to \eqref{axpsi} for $\alpha\in R^2$ the following properties for each $\alpha=e_i\in R^1$:

(1) for all $0<s\preccurlyeq(m_1,m_2)$
\begin{equation}\label{axpsi1}
\psi(\lambda, x-se_i)=\psi(\lambda,x+se_i)\qquad\text{ for}\ q^{x_i}=1\,;
\end{equation}

(2) for all $0<s\preccurlyeq(m_3,m_4)$
\begin{equation}\label{axpsi2}
\psi(\lambda, x-se_i)=\psi(\lambda,x+se_i)\qquad\text{ for}\ q^{x_i}=-1\,.
\end{equation}
Notice that for each $\alpha=e_i$ we get $1+\sum_{i=1}^4 m_i=2\mst_1$ different conditions.

\begin{definition}
A (nonzero) function $\psi(\lambda, x)$ with the properties
\eqref{psi}--\eqref{axpsi2} is called a {\bf Baker--Akhiezer (BA)
function} associated to $\{R,m\}$.
\end{definition}

\begin{remark}\label{lattice}
For a finite linear combination $f(x)=\sum_{\nu\in\V} a_\nu q^{\langle \nu, x\rangle}$, we call the {\bf support} of $f$ to be the convex hull of those points $\nu\in\V$ where $a_\nu\ne 0$. Then the property \eqref{psi} means that, for a fixed $\lambda$, the support of $\psi$ is contained in the set $\lambda+\nn$.
Moreover, in the ansatz \eqref{psi} for $\psi$ one can refine $P$ replacing it by any lattice $L$ containing $P$: that would still define the same object independent of the lattice.
Note also that in cases {\bf a} and {\bf b} the coefficients
$\psi_\nu$ in \eqref{psi} are nonzero only if $\nu\in\rho+Q$, cf.
\cite[Corollary 3.4]{C02}; this means that in the ansatz \eqref{psi}
in these cases one can replace $P$ by $Q$.
\end{remark}

\begin{theorem}\label{bau}[cf.\cite[Proposition 3.1 and Theorem 3.7]{C02}] A Baker--Akhiezer function $\psi(\lambda, x)$ exists and is unique up to multiplication by a factor depending on $\lambda$. As a function of $x$, $\psi$ is an eigenfunction of the Macdonald operators $D^\pi$ with $t=q^{-m}$ (or Koornwinder operator in case {\bf c}). Namely, we have
\begin{equation*}
    D^\pi \psi = \mathfrak m_\pi(\lambda)\psi\,,\qquad  \mathfrak m_\pi(\lambda)=\sum_{\tau\in W\pi}
q^{\langle\tau,\lambda\rangle}\,.
\end{equation*}
\end{theorem}

In case {\bf a} this was proved in \cite[Section 3]{C02}, while Section 6 of \cite{C02} also outlines the case {\bf c} (notice that the variables $(\lambda,x)$ are denoted as $(x,z)$ in \cite{C02}). The proof in case {\bf b} is the same; for the reader's convenience, we outline the main steps below.

\subsection{Outline of the proof of Theorem \ref{bau}}

First, one shows that, if exists, such a $\psi$ is unique up to a $\lambda$-depending factor. This follows the proof of Proposition 3.1 in \cite{C02}. The main idea is that the conditions \eqref{axpsi}--\eqref{axpsi2} lead to a linear system on the coefficients $\psi_\nu$, which has at most one-dimensional solution space when $\lambda$ is generic.

The most non-trivial part is to prove the existence of $\psi$; the next observation is the key. Let $\mathcal Q\subset \a$ denote the subspace of all $f(x)$ that have the same properties \eqref{axpsi}--\eqref{axpsi2} as $\psi$. Explicitly, in cases {\bf a} and {\bf b} this means that for every $\alpha\in R$ and $j=1,\dots ,m_\alpha$ a function $f\in\mathcal Q$ must satisfy the identities
\begin{equation}\label{quasi}
f\left(x+\frac12 j\alpha'\right)= f\left(x-\frac12 j\alpha'\right)\quad\text{when}\ \
q^{\langle\alpha,x\rangle}=1\,.
\end{equation}
In case {\bf c} \eqref{quasi} this should hold for $\alpha\in R^2$, while for $\alpha\in R^1$ one must have the identities as in  \eqref{axpsi1}--\eqref{axpsi2}.

It is easy to see that $\mathcal Q$ is a ring, with $\a^W\subset \mathcal Q$. We call $\mathcal Q$ the ring of {\it quasi-invariants}, associated to $(R, m)$; this is a $q$-analogue of the notion going back to \cite{CV,FV}. Then we have the following key result.

\begin{prop}\label{pres}[cf.\cite[Proposition 2.1]{C02}] Under the integrality assumptions of \ref{para}, each of the  operators $D^\pi$ described in Section \ref{mops} preserves the corresponding ring of quasi-invariants, i.e.
$D^\pi(\mathcal Q)\subseteq \mathcal Q$.
\end{prop}

Using Proposition \ref{pres}, one can construct a BA function $\psi$ by repeatedly applying $D^\pi$ to a suitable initial function, reducing its support at each application so that eventually the support lies within the polytope $\mathcal N$.
In case {\bf a} this is explained in detail in Sec. 3.2 of \cite{C02}, and everything applies with only minor changes to all three cases. To formulate the result, we need the function $Q(x)$ as follows:
\begin{equation}\label{QQ}
Q(x)=\Delta(x)\Delta(-x)\,,
\end{equation}
where $\Delta(x)$ is defined in \eqref{Deltam}--\eqref{Delta2}. It is trivial to see that any function divisible by $Q$ is quasi-invariant.

\begin{theorem}\label{exist}[cf.\cite[Theorem 3.7]{C02}]
Let $D^\pi$ be any of the operators defined in Section \ref{mops}.
Define $\psi(\lambda,x)$ as follows:
\begin{equation}\label{ber}
\psi=\prod_{\nu}\left(D^\pi-\mathfrak
m_\pi(\lambda+\nu)\right)\left[q^{\langle\lambda+\rho,x\rangle}Q(x) \right]\,,
\end{equation}
where $\mathfrak m_\pi$ are the orbitsums \eqref{orb}, $Q$ is as in \eqref{QQ}, and the
product is taken over all $\nu\ne 0$ having the form
$\nu=\sum_{\alpha\in R_+}l_\alpha\alpha$ with
$l_\alpha=0,\dots , m_\alpha$. Then

\noindent(i) $\psi$ has the required form \eqref{psi};

\noindent(ii) the coefficient $\psi_{-\rho}$ in its expansion
\eqref{psi} is nonzero and, therefore, $\psi$ is nonzero;

\noindent(iii) $\psi$ satisfies the conditions \eqref{axpsi}--\eqref{axpsi2};

\noindent(iv) as a function of $x$,\ $\psi$ is an eigenfunction of
the operator $D^\pi$, $D^\pi\psi=\mathfrak m_\pi(\lambda)\psi$.
\end{theorem}

The proof is completely parallel to the proof of \cite[Theorem 3.7]{C02}. This establishes the existence of $\psi$ and finishes the proof of Theorem \ref{bau}. \hfill$\Box$

\subsection{Rank one case}\label{rankone}

Consider the rank one case $R=A_1=\{\pm\alpha\}$ with $m_\alpha\in \Z_+$ denoted as $m$. In this situation the cases {\bf a} and {\bf b} are equivalent up to rescaling the variables. We identify $\V$ with $\R$ so that $\alpha=2$, and fix the scalar product on $\V$ by $\langle\alpha,\alpha\rangle=4$. The ring $\a$ is then the ring of Laurent polynomials in $z=q^{x}$.

Consider case {\bf b}, then we have parameters $q_\alpha=q^2$ and $t:=t_\alpha=q^{-2m}$.
The Macdonald operator corresponding to $\pi=1$ looks as follows:
\begin{equation}\label{mop1}
D=\frac{1-tq^{2x}}{t^{1/2}(1-q^{2x})}T^1+\frac{1-tq^{-2x}}{t^{1/2}(1-q^{-2x})}T^{-1}\,.
\end{equation}
We have $D(\mathcal Q)\subseteq \mathcal Q$, where the ring of quasi-invariants $\mathcal Q\subset\a$ consists of all $f(x)$ such that for $j=1,\dots, m$
\begin{equation}\label{qinv1}
f\left(x-j\right)= f\left(x+j\right)\quad\text{for}\ q^{2x}=1\,.
\end{equation}
It is easy to show that $\mathcal Q$ is generated by
\begin{equation}\label{gen1}
u=q^{x}+q^{-x}\quad\text{and}\quad v=\left(q^{x}-q^{-x}\right)\prod_{j=\pm 1}^{\pm m} \left(q^{x+j}-q^{-x-j}\right)\,.
\end{equation}
A BA function must have the form
\begin{equation}\label{psir1}
\psi(\lambda,x)=q^{(\lambda+m)x}\sum_{\nu=0}^{2m}\psi_\nu(\lambda)q^{-\nu x}\,,
\end{equation}
and satisfy the conditions \eqref{qinv1} in the $x$-variable. From that one calculates the coefficients $\psi_\nu$ and finds that $\psi_\nu=0$ for odd $\nu$, while the remaining coefficients can be chosen as follows:
\begin{equation}\label{coeff1}
\psi_0=1\,,\quad \psi_{2s}=\prod_{j=1}^s \frac{(t^{1/2}q^{j-1}-t^{-1/2}q^{-j+1})(t^{1/2}q^{j-1-\lambda}-t^{-1/2}q^{-j+1+\lambda})}
{(q^{j}-q^{-j})(q^{j-\lambda}-q^{-j+\lambda})}
\end{equation}
(see \cite[Section 4.1]{C02} for the details of such calculation.)

\medskip

The two generators \eqref{gen1} are related by
\begin{equation}
v^2=(u^2-4)\prod_{j=1}^m \left(u^2-2-q^{2j}-q^{-2j}\right)^2\,.
\end{equation}
The ring $\mathcal Q$ is therefore isomorphic to the coordinate ring of a singular hyperelliptic curve $\Gamma$ with $2m$ double points.

Restricting to $\lambda=n\in\Z$, we obtain a function
\begin{equation*}\label{asr1}
\psi(n, x)=z^n \sum_{\nu=-m}^{m}\psi_\nu(n)z^\nu\,,\quad z=q^x\,.
\end{equation*}
Since $\psi(n,x)\in\mathcal Q$ for all $n$, it can be viewed as a function $\psi(n, P)$ of $n\in\Z$ and $P\in\Gamma$. The affine curve $\Gamma$ can be completed by adding two points $P_-, P_+$, corresponding to $z=0$ and $z=\infty$, respectively. For $P$ near $P_{\pm}$ we have the following asymptotic formulas in terms of the local coordinate $z$:
\begin{equation}\label{ass}
\psi(n,P)=\psi_{\pm m}(n)z^{n\pm m}(1+{\mathrm O }(z^{\mp 1}))\quad\text{as}\quad P\to P_{\pm}\,.
\end{equation}

\subsection{Comparison with the BA functions in finite-gap theory}\label{fgap}

The above notion of a Baker--Akhiezer function in the rank one case should be compared to the Baker--Akhiezer functions that appear in the finite-gap approach to the Toda lattice equation. Namely, according to \cite{Kr2}, the relevant (stationary) Baker--Akhiezer function is determined by the following algebro-geometric data: a genus $g$ hyperelliptic curve $\Gamma$ given by $v^2=\prod_{i=1}^{2g+2}(u-u_i)$ and a non-special divisor $P_1,\dots, P_g$ on $\Gamma$. By definition, the BA function $\psi(n, P)$ is a meromorphic function on $\Gamma$, depending on the discrete variable $n\in\Z$, that has simple poles at $P_1,\dots,P_g$ and whose asymptotic behaviour near two `infinite' points $P_{\pm}$ in terms of the local coordinate $u\sim \infty$ looks as follows:
\begin{equation*}
 \psi^{\pm}(n, P)=\mu_n^{\pm}u^{\pm n}(1+\xi_1(n)u^{-1}+\dots)\,.
\end{equation*}
Such $\psi$ is uniquely defined up to an arbitrary $n$-dependent factor (this freedom is eliminated in \cite{Kr2} by assuming $\mu_n^+ \mu_n^-=1$). To compare this to \eqref{ass}, we use that $u\sim z^{\pm 1}$ near $P_{\pm}$. We conclude that the BA function considered above is a particular singular limit of the notion from \cite{Kr2}, with the curve $\Gamma$ having $2m$ double points and with the divisor $P_1+\dots+P_g$ replaced by $mP_++mP_-$. Note that the general setup allowing singular curves $\Gamma$ and sheaves instead of divisors was suggested in \cite{Mum}.

\medskip

In case {\bf c} the situation is similar. In this case $R=\{\pm 2\}$ as before, but now we have four parameters $m_1,\dots, m_4\in \frac12 \Z$ (the parameter $m_5$ drops out when $n=1$). The ring of quasi-invariants $\mathcal Q$ consists of all $f(x)\in\a$ such that
\begin{equation*}
f\left(x-s\right)= f\left(x+s\right)\quad\text{if}\quad \begin{cases}\, q^{x}=1\quad\text{and}\quad 0<s\preccurlyeq(m_1,m_2)\,,\\\,q^{x}=-1\quad\text{and}\quad 0<s\preccurlyeq(m_3,m_4)\,.\end{cases}
\end{equation*}
It is easy to see that $\mathcal Q$ is generated by $u=q^x+q^{-x}$ and
\begin{equation*}
v=\left(q^{x}-q^{-x}\right)v_-(x)v_+(x)
\end{equation*}
where
\begin{equation*}
v_-(x)=\prod_{0<s\preccurlyeq(m_1,m_2)}\left(q^{\frac{x-s}{2}}-q^{\frac{-x+s}{2}}\right)\left(q^{\frac{x+s}{2}}-q^{\frac{-x-s}{2}}\right)
\end{equation*}
and
\begin{equation*}
v_+(x)=\prod_{0<s\preccurlyeq(m_3,m_4)}
\left(q^{\frac{x-s}{2}}+q^{\frac{-x+s}{2}}\right)\left(q^{\frac{x+s}{2}}+q^{\frac{-x-s}{2}}\right)\,.
\end{equation*}
The two generators are related by the relation
\begin{equation}\label{gbc}
v^2=(u^2-4)\prod_{0<s\preccurlyeq(m_1,m_2)} \left(u-q^{s}-q^{-s}\right)^2\prod_{0<s\preccurlyeq(m_3,m_4)} \left(u+q^{s}+q^{-s}\right)^2\,.
\end{equation}
This defines a singular hyperelliptic curve $\Gamma$ with $2m_\alpha=1+\sum_{i=1}^4 m_i$ double points. The case of $R=A_1$ considered above can be viewed as a subcase of this, corresponding to $m_1=m_3=-\frac12$ and $m_2=m_4=m$.
When $m:=m_\alpha=\frac12(1+\sum_{i=1}^4 m_i)$ is integer, $\psi(\lambda,x)$ still has the form \eqref{psir1}, so specializing to $\lambda=n\in\Z$ we get Krichever's BA function for the singular curve \eqref{gbc}.
In the case when $m=m_\alpha$ is half-integer, the summation in \eqref{psir1} should be taken over half-integers. Thus, to get a function on $\Gamma$, one needs to restrict to $\lambda=\frac12+n$, $n\in\Z$. As a result, the divisor $P_1+\dots +P_g$ in this case has the form $(m-\frac12)P_++(m+\frac12)P_-$.

\subsection{Normalized BA function}

A BA function $\psi$ is not unique because one can multiply it by a $\lambda$-dependent factor. Let us use this freedom and prescribe the coefficient $\psi_{\nu}$ at one of the vertices of the polytope $\mathcal N$ to be as follows:
\begin{equation}
\label{norm0} \psi_{\rho}=\dst(\lambda)\,,
\end{equation}
where $\dst=\Delta_{\rst, \mst}$. Such a BA function is therefore uniquely defined. Note that the function \eqref{ber} does not satisfy the condition \eqref{norm0} in general.

\medskip
\begin{definition}
The {\bf normalized BA function} is the unique function $\psi(\lambda, x)$
with the properties \eqref{psi}--\eqref{axpsi2} and the normalization
\eqref{norm0}.
\end{definition}

This choice of normalization is justified by the following result, which in cases {\bf a} and {\bf c} was obtained in \cite[Sections 4 and 6]{C02}.

\begin{theorem}\label{no}
The normalized BA function $\psi$ has the following properties:

\noindent (i) for all $w\in W$ the coefficient $\psi_{w\rho}$ in \eqref{psi}
has the form
\begin{equation}
\label{norm}
\psi_{w\rho}=\dst (w^{-1}\lambda)\,;
\end{equation}

\medskip

\noindent (ii) $\psi(\lambda, x)$ can be presented in the form
\begin{equation}\label{psi1}
\psi = q^{\langle\lambda, x\rangle}\sum_
{\genfrac{}{}{0pt}{}{\nu\in\nn\cap\, \rho+P}{\nu'\in\nnst\cap\,\rho'+ P'}}
\psi_{\nu \nu'} q^{\langle\nu,x\rangle}q^{\langle\nu', \lambda\rangle}\,,
\end{equation}
with $\psi_{\nu\nu'}\in\mathbb Q(q^{1/2})$, where $\mathbb Q(q^{1/2})$ denotes the field extension of $\mathbb Q$ by all $q_\alpha^{1/2}$ with $\alpha\in R$;

\medskip

\noindent (iii) We have the following bispectral duality:
\begin{equation}\label{dual}
 \psi(\lambda, x)=\psist(x, \lambda)\,,
\end{equation}
where $\psist$ is the normalized BA function associated to $(\rst, \mst)$.

\end{theorem}

\begin{proof}
In case {\bf a} this follows from Proposition 4.4 and Theorem 4.7 of \cite{C02}. The cases {\bf b} and {\bf c} can be treated similarly (see Theorems 6.7, 6.8 of \cite{C02} in case {\bf c}). The statement that $\psi_{\nu \nu'}\in \mathbb Q(q^{1/2})$ is not mentioned in \cite{C02}, but it easily follows from the construction of $\psi$, see formula \eqref{ber} above.
\end{proof}

\medskip
Note that the duality \eqref{dual} implies that $\psi(\lambda, x)$ has the following
properties in the $\lambda$-variable: for each $\alpha\in R$ (or $\alpha\in R^2$ in case {\bf c}) and $j=1,\dots
,m_\alpha$
\begin{equation}\label{axxpsi}
\psi\left(\lambda+\frac 12 j\alpha, x\right)\equiv
\psi\left(\lambda-\frac 12 j\alpha, x\right)\quad\text{for}\ \
q^{\langle\alst,\lambda\rangle}=1\,,
\end{equation}
and, additionally in case {\bf c},

(1) for all $0<s\preccurlyeq(\mst_1,\mst_2)$
\begin{equation}\label{daxpsi1}
\psi(\lambda-se_i,x)=\psi(\lambda+se_i,x)\qquad\text{ for}\ q^{\lambda_i}=1\,;
\end{equation}

(2) for all $0<s\preccurlyeq(\mst_3,\mst_4)$
\begin{equation}\label{daxpsi2}
\psi(\lambda-se_i,x)=\psi(\lambda+se_i,x)\qquad\text{ for}\ q^{\lambda_i}=-1\,.
\end{equation}

\medskip

Part $(i)$ of the above theorem, together with uniqueness of $\psi$, implies the following symmetries of $\psi$.

\begin{lemma}\label{w}
The normalized BA function has the following invariance properties:

\noindent $(i)\quad \psi(w\lambda, wx)=\psi(\lambda, x)$ for any $w\in W$;

\noindent $(ii)\ \psi(-\lambda, -x)=\psi(\lambda, x)$.

\noindent $(iii)\ \psi(\lambda, x; q^{-1})=\psi(\lambda, -x; q)$.
\end{lemma}

For the proof of the first part, see \cite[Lemma 5.4]{C02}. Parts $(ii)$ and $(iii)$ are proved similarly. \hfill$\Box$

\begin{remark}
In the rank one case $R=A_1$, one can express $\psi$ in terms of the basic hypergeometric series $_2\phi_1(a,b;c;q,z)$, which reduces the properties \eqref{axpsi} and the statements of Theorem \ref{no} to the known identities for $_2\phi_1$, see \cite{Ko1}. Other expressions for $\psi$ in the rank-one case exist \cite{R1,EV1,St2}. In higher rank, for $R=A_n$ a function closely related to our $\psi$ was constructed in \cite{FVa} via a version of Bethe ansatz, and in \cite{ES} via representation theory of quantum groups.
\end{remark}

\begin{remark}\label{series}
From \eqref{nu}, \eqref{psi} it follows that $\psi$ can be presented in the form
\begin{equation}\label{expand}
\psi(\lambda, x) = q^{\langle\lambda+\rho,x\rangle}\sum_{\nu\in P_-} \Gamma_\nu(\lambda) q^{\langle\nu,x\rangle}\,,
\end{equation}
where $P_-:=-P_+$ and the sum is finite. (In cases {\bf a} and {\bf
b} the summation effectively takes place  over $\nu\in Q_-\subseteq
P_-$, cf. Remark \ref{lattice}.) The leading coefficient $\Gamma_0$
can be determined from \eqref{norm} as
\begin{equation}\label{gamma0}
 \Gamma_0= \dst(\lambda)\,.
\end{equation}
Recall that this $\psi$ is an eigenfunction of Macdonald difference operators with $t=q^{-m}$. For generic $t$ the eigenfunctions are no longer given by finite sums, but rather infinite series of the form \eqref{expand}. Such infinite series solutions were studied in \cite{LS}, \cite{SvM}, \cite{vM}. The fact that for $t=q^{-m}$ with (half-)integer $m$ those series terminate is non-obvious, but it follows from the above results and the uniqueness of the formal series solution, cf. \cite[Proposition 4.13]{LS}. Note also that for $t=q^{m+1}$ the series solutions \eqref{expand} are no longer finite, but are in fact still elementary functions.
\end{remark}

\subsection{Roots of unity}\label{unity}
The proofs of the above results in \cite{C02} require $q$ not being a root of unity; this is needed for the proof of the crucial Lemma 3.2 of \cite{C02}. In fact, for given multiplicities $m$ one has to avoid only certain roots of unity. Namely, let us assume that the function $\Delta$ defined by \eqref{Deltam}--\eqref{Delta2} has simple zeroes, i.e. all the factors are distinct.
Explicitly, in case {\bf a} and {\bf b} this means that for all $\alpha\in R$
\begin{equation}\label{roots}
q_\alpha^{j}\ne 1\quad\text{for}\ j=1,\dots, m_\alpha-1\,.
\end{equation}
In case {\bf c} our assumption is that
\begin{equation}\label{roots1}
q^{j}\ne 1\quad\text{for}\ j=1,\dots, m_5-1\,,
\end{equation}
and that the following numbers are pairwise distinct:
\begin{equation}\label{roots2}
q^{s}\ \text{with}\ 0<s\preccurlyeq (m_1,m_2)\quad\text{and}\ \, -q^{s}\ \text{with}\ 0<s\preccurlyeq (m_3,m_4)\,.
\end{equation}

\begin{prop}
With the conditions \eqref{roots}--\eqref{roots2} the statements of Theorems \ref{bau}, \ref{no} remain valid.
\end{prop}

\begin{proof}
The conditions \eqref{axpsi}--\eqref{axpsi2} are equivalent to an overdetermined linear system for coefficients $\psi_\nu$, see the proof of \cite[Lemma 3.2]{C02}. We know that $\psi$ exists for generic $q$, therefore, it must exist for any $q$. Its uniqueness is based on \cite[Lemma 3.2]{C02} and elementary geometric arguments. Looking at the proof of that lemma in \cite{C02}, given in case {\bf a}, one sees that it only requires an assumption that $q^j\ne 1$ for $j=1,\dots,m_\alpha-1$. In case {\bf b} it should be replaced by \eqref{roots}. In case {\bf c} everything is analogous for the roots $\alpha\in R^2$, which gives \eqref{roots1}. Finally, one needs to look at the corresponding linear system for $\alpha=e_i\in R^1$. In that case one can see, similarly to case {\bf a}, that in the limit $q^{\lambda_i}\to\infty$ this system has the matrix of coefficients being the Vandermonde matrix built from the numbers appearing in \eqref{roots2}. Therefore, the system has only zero solution provided that these numbers are pairwise distinct. This proves that conditions \eqref{roots1}--\eqref{roots2} are sufficient for the uniqueness of $\psi$ in case {\bf c}.
\end{proof}

\begin{remark} If one is interested in eigenfunctions of the difference operators $D^\pi$, then the assumptions \eqref{roots}--\eqref{roots2} are not very restrictive. Indeed, a quick look at the formula \eqref{m2} for the coefficients of the Macdonald operator in case {\bf b} shows that if $q_\alpha$ is a primitive $n$th root of unity then $m_\alpha$ can be reduced modulo $n$ as this does not change $t_\alpha=q^{-m_\alpha}$. Therefore, we can always assume that $m_\alpha<n$, and in that case \eqref{roots} is automatic. The situation in cases {\bf a} and {\bf c} is similar. Thus, the Macdonald operators with $q=t^{-m}$ with integral $m$ will always have BA functions $\psi(\lambda,x)$ as their eigenfunctions, for any $q\in\c^\times$. For fixed $t=q^{-m}$ these eigenfunctions
are analytic in $q$ provided \eqref{roots}--\eqref{roots2}.
\end{remark}

\subsection{Generalized Weyl formula}\label{we}

Let us explain, following \cite{C02}, the relationship between $\psi(\lambda, x)$ and Macdonald polynomials $p_\lambda$. Given the normalized BA function $\psi(\lambda, x)$, we consider two functions
$\Phi_\pm$ obtained by (anti)symmetrization in $\lambda$:
\begin{equation}\label{phi}
\Phi_+(\lambda, x)=\sum_{w\in W}\psi(w\lambda, x)\,,\quad \Phi_-(\lambda, x)=\sum_{w\in
W}(-1)^w\psi(w\lambda, x)\,.
\end{equation}
Note that (anti)symmetrization in $x$ would give the same result,
due to Lemma \ref{w}; hence, $\Phi_+$ is $W$-symmetric in $x$, and $\Phi_-$ is antisymmetric.

Introduce the following function:
\begin{equation}\label{delta}
\delta(x)=\Delta(x)\Delta(-x)\delta_0(x)\,,
\end{equation}
where
\[
\delta_0(x)=\prod_{\alpha\in R_+}\left(q^{\langle\alpha, x\rangle/2}-q^{-\langle\alpha, x\rangle/2} \right)\,.
\]
Recall the vector $\rho$ \eqref{rho} and let
\begin{equation}\label{wrho}
\wrho=\frac12\sum_{\alpha\in R_+} (m_\alpha+1)\alpha\,.
\end{equation}

\medskip

\begin{theorem}[cf. {\cite[Theorem 5.11]{C02}}]\label{weyl} For $\lambda\in \wrho+P_{++}$ we have
\begin{gather}\label{weyl2}
\Phi_-(\lambda, x)=\dst(\lambda)\delta(x)\,
p_{\lambda-\wrho}(x;q,q^{m+1})
\intertext{and}
\label{weyl1}
\Phi_+(\lambda, x)=\dst(\lambda)\,p_{\lambda+\rho}(x;q,q^{-m}) \,.
\end{gather}
Note that the condition $\lambda\in \wrho+P_{++}$ ensures that $\dst(\lambda)\ne 0$.
\end{theorem}

\begin{proof}
In case {\bf a} this is \cite[Theorem 5.11]{C02}, and the same proof works in cases ${\bf b}$ and ${\bf c}$.
The proof of \eqref{weyl2} goes as follows: firstly, one shows that every antisymmetric quasi-invariant is divisible by $\delta(x)$. This proves that $\Phi_-=\delta(x)f(x)$ for some $f\in\a^W$. Next, since $\Phi_-$ is an eigenfunction of the Macdonald operators with $t=q^{-m}$, a simple computation shows that $f$ must be an eigenfunction of $D^\pi$ with $t=q^{m+1}$. Finally, by comparing the leading terms and using Theorem \ref{ueig}, we get that $f$ is proportional to the required Macdonald polynomial, see \cite{C02} for the details.

The proof of \eqref{weyl1} is similar. By construction, the left-hand side of \eqref{weyl1} is a $W$-invariant polynomial eigenfunction of the Macdonald operators with $t=q^{-m}$, with the required leading exponential term $\dst(\lambda)q^{\langle\lambda+\rho,x\rangle}$ that comes from $\psi(x,\lambda)$. We will need a slightly stronger result: $\psi$, and therefore $\Phi_+$, is a common eigenfunction for the full family of Cherednik--Macdonald operators with $t=q^{-m}$. This follows from the results of \cite[Sections 5.1, 5.2]{C02}; alternatively, one can derive that by using that $\psi(\lambda,x)$ is a specialisation of the formal $q$-Harish-Chandra series, cf. Remark \ref{series}.

Thus, the left-hand side of \eqref{weyl1} is a $W$-invariant eigenfunction for the Cherednik--Macdonald operators.
By Remarks \ref{ssu}, \ref{ssuu}, to check that it coincides with the relevant Macdonald polynomial, it suffices to check that for $\lambda\in\wrho+P_{++}$ and $\mu\in P_{++}$ we have
\begin{equation}\label{ww}
 w\lambda\ne \mu-\rho\quad\text{for}\ w\ne 1\,.
\end{equation}
The assumption on $\lambda$ gives that $\langle \lambda, \alpha^\vee\rangle > m_\alpha$ for any $\alpha\in R_+$. For  $w\ne 1$ we can find a simple positive coroot $\beta^\vee$ such that $\alpha^\vee:=-w^{-1}\beta^\vee$ is also positive. Then we have $\langle\beta^\vee,w\lambda\rangle=\langle w^{-1}\beta^\vee,\lambda\rangle=-\langle \alpha^\vee, \lambda\rangle<-m_\alpha=-m_\beta$.
This implies that $w\lambda\notin -\rho + P_{++}$, which proves \eqref{ww}.
\end{proof}

\begin{remark}\label{small}
It follows from Proposition \ref{eval1} below, that for sufficiently small weights there is a formula similar to \eqref{weyl1}. Namely, the function $\psi(\lambda,x)$ is $W$-invariant if $\lambda\in -\rho+P_{++}\cap -P_{++}$, in which case we have $\psi(\lambda,x)=\dst(\lambda)p_{\lambda+\rho}(x;q,q^{-m})$.
Proof is similar, only in this case it suffices to check \eqref{ww} for $\mu-\rho<\lambda$, and this is obvious because now $\lambda\in -P_{++}$.
\end{remark}

\begin{remark}
For $m=0$ $\psi(\lambda, x)$ is simply $q^{\langle\lambda,x\rangle}$ while $p_\lambda(x; q, q)$ are the characters of the corresponding Lie algebra of type $R$. Thus, for $m=0$ formula \eqref{weyl2} turns into the classical Weyl character formula. In case $R=A_n$ formula \eqref{weyl2} was conjectured by Felder and Varchenko \cite{FVa} and proved by Etingof and Styrkas \cite{ES}. We note that the evaluation and duality identities for $p_\lambda$ are trivial consequences of this formula and the duality \eqref{dual}, see \cite[Section 5.5]{C02} for the details.
\end{remark}

\begin{remark}
The above proof of \eqref{weyl2} uses crucially the characterisation of $p_\lambda$ as polynomial eigenfunctions for $D^\pi$. There is an alternative way to prove \eqref{weyl2} using the orthogonality relations for BA functions. As a by-product, this gives an alternative way to establish the existence of $p_\lambda$, see Remark \ref{macpex} below.
\end{remark}

\begin{remark}
As indicated in Remark \ref{series}, the BA function $\psi(x,\lambda)$ can be thought of as a specialization at $t=q^{-m}$ of the asymptotically free eigenfunctions given by the $q$-Harish-Chandra series \cite{LS}. It is known that for special values of $\lambda$ these series become $W$-invariant and reduce to Macdonald polynomials. However, in the above theorem the corresponding $\psi(\lambda,x)$ is not $W$-invariant and further (anti)-symmetrization is needed. This seems a contradiction, but in fact the asymptotically free eigenfunctions are not well defined for those specific $\lambda$ and $t$. This is best illustrated in the rank-one case. We use the notation of Sec.~\ref{rankone}, so $R=A_1=\{\pm 2\}$ with $q_\alpha=q^2$ and $t=t_\alpha=q^{-2m}$. In this case the asymptotically free eigenfunctions are given by $\phi(x,\lambda)=q^{(\lambda+m)x}\sum_{\nu\in 2\Z_+} \phi_\nu(\lambda) q^{-\nu x}$, where the coefficients are given by \eqref{coeff1}. From these formulas it is clear that the series is well-defined for generic $\lambda$ and  arbitrary $t$. When $t=q^{-2m}$ with $m\in\Z_+$, the series terminates, reducing to a BA function $\psi(x,\lambda)$ (normalized by $\psi_\rho=1$). On the other hand, when $\lambda+m\in\Z_+$, the series also terminates (if we assume that $m$ is generic in order to avoid the poles in $\phi_{2s}$). However, if we assume that $m\in\Z$ and $\lambda+m\in\Z_+$ simultaneously, then the series is not well-defined (there will be zero factors appearing both in the numerator and denominator of $\phi_{2s}$). That is why for such $\lambda, t$ we may have a symmetric eigenfunction together with a non-symmetric one, both asymptotically free. Note that this happens even for $t=q^{2m+2}$ with $m\in\Z_+$, when the Macdonald polynomials are perfectly well-defined. In that case the asymptotically free solution that is valid for generic $\lambda$ is given by $\phi(x,\lambda)=\delta^{-1}(x)\psi(x,\lambda)$, where $\delta$ is as in \eqref{delta}. Such $\phi$ can be expanded into an infinite series, and the formula \eqref{weyl2} tells us that a suitable combination of these infinite series becomes a finite sum, reducing to the appropriate Macdonald polynomial.
\end{remark}

\subsection{Evaluation}\label{seva}

Relation \ref{weyl1} gives a well-defined expression for $p_\lambda$ only if $\lambda\in \wrho+\rho+P_{++}$, i.e. if $\lambda$ is sufficiently large. This reflects the fact that for $m_\alpha\in\Z_+$, some of $p_\lambda$
are not well-defined, cf. \cite[Corollary 5.13]{C02}. This is also related to the fact that while for generic $\lambda$ the function $\psi(\lambda,x)$ has the support $\lambda+\nn$, for special $\lambda$ the support becomes smaller. In particular, the support can reduce to a single point, as the following proposition shows.

\begin{prop}\label{eval} For $\lambda=w\rho$, $w\in W$, the normalized BA function $\psi(\lambda,x)$ does not depend on $x$ and is equal to $\dst(-\rho)\ne 0$.
\end{prop}

\begin{proof} The vectors $w\rho$ point to the vertices of the polytope $\nn$. Each vertex corresponds to a choice of a positive half $R_+\subset R$, and for any two adjacent vertices $\lambda_1$, $\lambda_2$ we have $\lambda_2=s_\alpha(\lambda_1)$ and $\lambda_2=\lambda_1-m_\alpha\alpha$ for a suitable $\alpha\in R$.
Put $\lambda=\frac12 (\lambda_1+\lambda_2)$. Then $\langle \alpha, \lambda\rangle =0$ so that $q^{\langle\alst,\lambda\rangle}=1$; also $\lambda_{1,2}=\lambda\pm\frac12 m_\alpha\alpha$. Therefore, applying \eqref{axxpsi} with $j=m_\alpha$ gives us that $\psi(\lambda_1,x)=\psi(\lambda_2,x)$ in cases {\bf a}, {\bf b}, or {\bf c} with $\alpha\in R^2$. In the remaining case $\alpha= 2e_i\in 2R^1$ this also works, because in that case we can use \eqref{daxpsi1} for $s=\mst_1$:
\[
\psi(\lambda-\mst_ie_i,x)=\psi(\lambda+\mst_ie_i,x)\,.
\]
According to \eqref{cmalpha}, we have $\mst_ie_i=\frac12 m_\alpha\alpha$. Therefore, in that case we also obtain that $\psi(\lambda_1,x)=\psi(\lambda_2,x)$.

So, in all cases we obtain that the functions $\psi(w\rho,x)$ with $w\in W$ are all the same.
Each of these functions has support within $w\rho+\nn$. Since $$\cap_{w\in W}\{w\rho+\nn\}=\{0\}\,,$$ they all are constants. To evaluate this constant, we need to look at the coefficient $\psi_{\rho}(-\rho)$ in \eqref{psi}, which equals $\dst(-\rho)$ by \eqref{norm0}. The fact that this is nonzero is easy to check.
\end{proof}

\begin{remark}\label{evadual}
By duality \eqref{dual}, we also have $\psi(\lambda,w\rhost)=\Delta(-\rhost)$ for all $w\in W$. In particular, for $\lambda=\rho$ and $x=\rhost$ this gives $\psi(\rho,\rhost)=\Delta(-\rhost)=\dst(-\rho)$.
\end{remark}

More generally, we have the following result, which reduces to Proposition \ref{eval} in the case $\mu=0$.

\begin{prop}\label{eval1}
Let $\mu\in P_{++}$ be such that $\langle\alpha^\vee,\mu\rangle\le m_\alpha$ for every {\it simple} root $\alpha\in R_+$, and $\lambda:=\mu-\rho$. Then $\psi(\lambda,x)\in \A^W$ and we have
\begin{equation*}
\psi(\lambda,x)=\sum_{\nu\le\mu}a_{\mu\nu}\mathfrak m_\nu\,,\qquad a_{\mu\mu}=\dst(\lambda)\ne 0\,.
\end{equation*}
Here $\nu\le\mu$ denotes the same partial ordering on $P_{++}$ as in \eqref{macp}.
\end{prop}

\begin{proof}
For any simple root $\alpha\in R_+$ we have $\langle\alpha^\vee,\lambda\rangle=\langle\alpha^\vee,\mu\rangle-m_\alpha$, therefore, $-m_\alpha\le \langle\alpha^\vee,\lambda\rangle\le 0$. Put $\lambda_1=\lambda$ and $\lambda_2=s_\alpha\lambda$, where $s_\alpha$ is the corresponding simple reflection. Then the same argument as above shows that $\psi(\lambda_1,x)=\psi(\lambda_2,x)$, i.e. $\psi(\lambda,x)=\psi(s_\alpha\lambda,x)$. By Lemma \ref{w}$(i)$, we have $\psi(\lambda,x)=\psi(\lambda,s_\alpha x)$. Since this applies to every simple reflection, we conclude that $\psi(\lambda,x)$ is $W$-invariant. The support of $\psi(\lambda,x)$ is contained in $\lambda+\nn$, thus $\psi$ must be a combination of orbitsums $\mathfrak m_\nu$ with $\nu\le\lambda+\rho=\mu$. The leading coefficient $a_{\mu\mu}$ can be found as $\psi_{\rho}(\lambda)$, which equals $\dst(\lambda)$. Since $\langle\alst,\lambda\rangle\le 0$ for all simple roots, we have $\dst(\lambda)\ne 0$.
\end{proof}

As mentioned earlier, in the case $t=q^{-m}$ some of the Macdonald polynomials $p_\lambda$ do not exist. We have seen that there are two types of Macdonald polynomials that exist for $t=q^{-m}$, namely, $p_\lambda$ with large $\lambda$ as in \eqref{weyl1}, or $p_\mu$ with small $\mu$ as in Proposition \ref{eval1}. It is interesting to note that, as the formula \eqref{ort1} below shows, $p_\lambda$'s have positive norms, while the norms of $p_\mu$ are all zero.

\begin{remark}\label{evalcompare}
The result of Proposition \ref{eval} can be viewed as a counterpart of the evaluation formula for $p_\lambda$.
Indeed, let us substitute $x=-\rhost$ into \eqref{weyl1}. Using Remark \ref{evadual}, we get that $\psi(w\lambda,-\rhost)=\dst(-\rho)$ for all $w\in W$. As a result, \eqref{weyl1} gives us that
\[
|W|\dst(-\rho)=\dst(\lambda)\,p_{\lambda+\rho}(-\rhost;q,q^{-m})\,,
\]
provided $\lambda$ sufficiently large so that $p_{\lambda+\rho}$ are well-defined. Denoting $\mu=\lambda+\rho$ and using the notation $\rho_m$, $\rhost_m$ for vectors \eqref{rho}, we get
\begin{equation}\label{evam}
p_{\mu}(\rhost_{-m};q,q^{-m})=|W|\frac{\dst(\rho_{-m})}{\dst(\mu+\rho_{-m})}\,.
\end{equation}
This should be compared with the evaluation identity for $p_\lambda$, see \cite{Ch2} in cases {\bf a}, {\bf b}, or \cite{M4} for all three cases. In fact, formula \eqref{evam} can be obtained from the formula \cite[(5.3.12)]{M4} for $p_\lambda(\rhost_k;q,q^k)$ by analytic continuation in $k$, assuming the existence of $p_\lambda(x;q,q^{-m})$.

\end{remark}

\section{Orthogonality relations for BA functions}\label{ort}

Let us say that $\xi\in\V$ is {\bf big} if $|\langle\alpha, \xi\rangle|\gg 1$ for all $\alpha\in R$; more precisely, we will require that
\begin{equation}\label{big}
 |\langle\alpha^\vee, \xi\rangle|>m_\alpha\quad\text{for all}\ \alpha\in R\,.
\end{equation}

Let $C_\xi=\xi+i\V$ be the imaginary subspace in $\v$; it is invariant under translations by $\kappa Q^\vee$. Let $dx$ denote the translation invariant measure on $C_\xi$ normalized by the condition
\begin{equation*}
    \int_{C_\xi/\kappa Q^\vee} dx = 1\,.
\end{equation*}

\begin{theorem}\label{ort1}  For any $\lambda, \mu\in \V$ with $\lambda-\mu\in P$ and any big $\xi\in\V$ we have
\begin{equation}\label{ortrel}
  \int_{C_\xi/\kappa Q^\vee}\frac{\psi(\lambda,x)\psi(\mu, -x)}{\Delta(x)\Delta(-x)}\,dx=\delta_{\lambda, \mu}(-1)^{M}\dst(\lambda)\dst(-\lambda)\,,
\end{equation}
where $\delta_{\lambda, \mu}$ is the Kronecker delta and $M=\sum_{\alpha\in R_+}m_\alpha$.
\end{theorem}

\begin{proof}
The condition $\lambda-\mu\in P$ guarantees that $\psi(\lambda, x)\psi(\mu ,-x)$ is periodic in $x$ with respect to the lattice $\kappa Q^\vee$, thus, the integral is well-defined. The proof of the theorem rests on the following result.

\begin{prop}[cf. {\cite[Theorem 5.1]{EV}}]\label{res}
Let $I(\xi)$ denote the integral in the left-hand side of \eqref{ortrel}. Then $I(\xi)$ does not depend on $\xi$ provided it is big in the sense of \eqref{big}.
\end{prop}

The proof of the proposition occupies the next section. Assuming it, we can evaluate the integral by taking the limit $\xi\to \infty$ in a suitable Weyl chamber. Indeed, let us assume that $\xi$ stays deep inside the negative Weyl chamber, i.e. $\langle\alpha, \xi\rangle\ll 0$ for every $\alpha\in R_+$. In that case
\begin{equation*}
\mathrm{Re}\langle\alpha, x\rangle=\langle\alpha, \xi\rangle\ll 0\quad\text{for any}\ x\in\xi+i\V\,,
\end{equation*}
hence $\left|q^{-\langle\alpha,x\rangle}\right|\ll 1$. The properties \eqref{expand}--\eqref{gamma0} give us the asymptotic behaviour of $\psi$ for $x\in C_\xi$ as $\xi\to\infty$ inside the negative Weyl chamber:
\begin{gather*}
    \psi(\lambda, x)\sim \dst(\lambda)q^{\langle\lambda+\rho,x\rangle}\quad\text{and}\\
    \psi(\mu, -x)=\psi(-\mu, x)\sim \dst(-\mu)q^{\langle-\mu+\rho,x\rangle}\,.
\end{gather*}
For those $x$ we also have
\begin{equation*}
 {\Delta(x)\Delta(-x)}\sim (-1)^{\sum_{\alpha\in R_+}m_\alpha}q^{2\langle\rho,x\rangle}\,.
\end{equation*}
As a result, the asymptotic value of the integrand is
\begin{equation*}
   (-1)^M \dst(\lambda)\dst(-\mu) q^{\langle\lambda-\mu,x\rangle}\,.
\end{equation*}
In the case $\mu=\lambda$ this immediately leads to \eqref{ortrel}. On the other hand, when $\mu-\lambda$ is dominant
the integrand tends to zero as $\xi\to\infty$ in the negative chamber, thus the integral must vanish. Finally, by switching to another Weyl chamber one obtains the same result in the general case.
\end{proof}

\subsection{Proof of Proposition \ref{res}}

The proof is parallel to the proof of \cite[Theorem 5.1]{EV}. Let us first demonstrate the idea in the rank-one case of $R=A_1=\{\alpha, -\alpha\}\subset \R$, $Q=\Z\alpha$, $P=\frac12 Q$. In that case the integrand in \eqref{ortrel} is a meromorphic function of a single complex variable $x\in\c$, periodic with the period $\kappa\alpha^\vee$; we denote the integrand as $F(x)$. Thus, we have
\begin{equation*}
    I(\xi)=\int_{\xi}^{\xi+\kappa\alpha^\vee}F(x)\,dx\,.
\end{equation*}
To prove that $I(\xi)=I(\xi')$, we need to look at the residues of $F$ between the lines $\mathrm{Re}(x)=\xi$ and $\mathrm{Re}(x)=\xi'$. The integrand has simple poles at points where $q^{\langle \alpha, x \rangle} = q_\alpha^{\pm j}$ with $j=1, 2,\dots, m_\alpha$. These poles are naturally organized in groups, with $2m_\alpha$ poles in each group. Namely, for any $y$ such that $q^{\langle \alpha, y \rangle}=1$, we have $2m_\alpha$ poles of $F$ at
\begin{equation}\label{yj}
x=y_{\pm j}:=y\pm\frac12 j\alst\quad\text{with}\quad j=1,\dots, m_\alpha\,.
\end{equation}
The requirement that $\xi$ is big is equivalent to saying that these poles lie on one side of the line $\mathrm{Re}(x)=\xi$. We need to check that $I(\xi)=I(-\xi')$ for $\xi,\xi'\gg 0$. For that it is sufficient to check that the sum of the residues of $F$ at the points \eqref{yj} equals zero.

From \eqref{axpsi} we have
\begin{equation*}
    \psi(\lambda, y_{-j})=\psi(\lambda, y_{j})\qquad\forall\ j=1,2,\dots, m_\alpha\,,
\end{equation*}
and the same for $\psi(\mu,x)$. Also, it is clear that $\Delta(x)\Delta(-x)$ is invariant under the group $\{\pm 1\}\ltimes \kappa\Z\alpha^\vee$, which is isomorphic to the affine Weyl group of $R=A_1$. From that it easily follows that
\begin{equation*}
    \mathrm{res}_{x=y_{-j}}\,F(x)=-\mathrm{res}_{x=y_{j}}\,F(x)\qquad\forall\ j=1,2,\dots, m_\alpha\,.
\end{equation*}
Thus, the sum of the residues is indeed zero, and we are done.

The higher rank case is similar. We will give a proof for cases {\bf a} and {\bf b}; case {\bf c} is entirely similar.
For future application to deformed root systems \cite{C11}, let us make most of our arguments independent of the properties of root systems. Thus, we will only assume that $P$ and $Q^\vee$ are full rank lattices in $\V$, with $R\subset P$ and with $Q^\vee$ contained in the dual to $P$, i.e. with $\langle P, Q^\vee \rangle\subset \Z$.

The hyperplanes $\langle\alpha,x\rangle=0$ with $\alpha\in R$ separate $\V$ into several connected regions (chambers). Clearly, $I(\xi)$ does not change when $\xi$ stays within a particular chamber while remaining big. To show that the value of the integral is the same for every chamber, it is enough to check that $I(\xi)=I(\xi')$ when $\xi$ and $\xi'$ belong to adjacent chambers. Suppose that the two chambers are separated by the hyperplane $\langle \alpha, x \rangle =0$ for some $\alpha\in R$. Without loss of generality, we may assume that $\xi'=s_\alpha\xi$, with $\langle \alpha^\vee, \xi \rangle> m_\alpha$. Moreover, we can move $\xi$ and $\xi'$ inside the chambers to achieve that
\begin{equation}\label{bigg}
|\langle \beta, \xi\rangle|\gg |\langle \alpha, \xi\rangle|\quad\text{for all}\ \beta\ne \pm\alpha\ \text{in}\ R\,,
\end{equation}
and the same for $\xi'$.

The integral over $C_\xi/\kappa Q^\vee$ can be computed by integrating over any (bounded, measurable) fundamental region for the action of $\kappa Q^\vee$ on $C_\xi$. For example, we can choose a basis $\{\epsilon_1,\dots, \epsilon_n\}$ of $Q^\vee$ and integrate over the set of $x\in\v$ of the form
\begin{equation}\label{fr}
x(t_1,\dots, t_n)=\xi+\kappa\sum_{i=1}^n t_i\epsilon_i\,,\qquad t_i\in(0,1)\,.
\end{equation}
Moreover, one can replace $\epsilon_i$ by $\epsilon'_i=\sum a_{ij}\epsilon_j$ where the matrix $A=(a_{ij})$ is upper-triangular with $a_{ii}=1$: it is easy to see that the set \eqref{fr} for $\{\epsilon'_i\}$ will still be a fundamental region. (Note that the entries of $A$ do not have to be integers, so $\epsilon'_i$ may not belong to $Q^\vee$.) Using this, we can change the direction of $\epsilon_1$ arbitrarily; we will assume that $\epsilon_1$ is parallel to the above $\alpha$.

Up to an irrelevant constant factor we have $dx=dt_1\dots dt_n$ and
\begin{equation*}
 I(\xi)=\int F(x)\,dt_1\ldots dt_n\,,\qquad x=x(t_1,\dots, t_n)\,,
\end{equation*}
with
\begin{equation}\label{integrand}
F(x)= \frac{\psi(\lambda,x)\psi(\mu, -x)}{\Delta(x)\Delta(-x)}\,.
\end{equation}
For $I(\xi')$ we have a similar formula
\begin{equation*}
 I(\xi')=\int F(x')\,dt_1\ldots dt_n\,,\quad x'(t_1,\dots, t_n)=\xi'+\kappa\sum_{i=1}^n t_i\epsilon_i\,.
\end{equation*}
Both integrals can be computed by repeated integration. Therefore, to prove that $I(\xi)=I(\xi')$ it suffices to check that for any $t_2,\dots, t_n\in\R$ we have
\begin{equation}\label{int}
    \int_0^1 F(x)\,dt_1=\int_0^1 F(x')\,dt_1\,.
\end{equation}
Since $\epsilon_1$ is parallel to $\alpha$, the variable $x$ in the first integral moves in the direction of $\kappa\alpha^\vee$ through the point
\begin{equation*}
y=\xi+\kappa\sum_{i=2}^n t_i\epsilon_i\,.
\end{equation*}
Similarly, $x'$ in the second integral moves in the same direction through the point
\begin{equation*}
y'=\xi'+\kappa\sum_{i=2}^n t_i\epsilon_i\,.
\end{equation*}
Since $y-y'=\xi-\xi'=\xi-s_\alpha\xi=\langle \alpha^\vee, \xi \rangle\alpha$, the integration takes place along two parallel lines in the complex plane $\{y+z\alpha'\,|\,z\in\c\}$, which makes the situation similar to the rank-one case above. Namely, if we denote by $L$ and $L'$ the above two lines through $y$ and $y'$ then the relation \eqref{int} is equivalent to
\begin{equation}\label{int1}
\int_{L/\kappa \Z \alpha^\vee}F(y+z\alpha')\,dz=\int_{L'/\kappa\Z\alpha^\vee}F(y+z\alpha')\,dz\,.
\end{equation}
We therefore need to look at the poles of $F(y+z\alst)$ as a function of $z\in\c$. The poles between $L$ and $L'$ are those where
\begin{equation}\label{poles}
q^{\langle \alpha, y+z\alst \pm\frac12 j\alst \rangle} = 1\quad\text{with $j=1,2,\dots, m_\alpha$}\,.    \end{equation}
Other factors in $\Delta(x)\Delta(-x)$ will not contribute because of the assumption \eqref{bigg} and the fact that $y\in \xi+i\V$.

Similarly to the rank-one case, the poles \eqref{poles} are organized into groups with $2m_\alpha$ poles in each group. Namely, by a suitable shift in the $z$-variable, we can always make $q^{\langle \alpha, y \rangle} =1$ in such a way that the poles \eqref{poles} will correspond to $z=\pm\frac12 j$ {with} $j=1,\dots, m_\alpha$.
Now everything boils down to the following property of the integrand \eqref{integrand}.

\begin{lemma}\label{cancel} For any $x\in\v$ with $q^{\langle\alpha, x\rangle}=1$ and for all $j=1,\dots, m_\alpha$ we have
\begin{equation}\label{residues}
 \mathrm{res}_{z=-j/2}\,f(z)+ \mathrm{res}_{z=j/2}f(z)=0\,,\quad\text{where}\ f(z):=F(x+z\alst)\,.
\end{equation}
\end{lemma}

The lemma can be proved in the same manner as in the rank-one case, by using the properties \eqref{axxpsi} and the invariance of $\Delta(x)\Delta(-x)$ under the group $W\ltimes \kappa Q^\vee$. \hfill$\Box$

Using the lemma, we conclude that the relation \eqref{int1} is valid, and this finishes the proof of Proposition \ref{res}. \hfill$\Box$

\subsection{Norm identity for Macdonald polynomials}\label{norms}

Let us keep the notation of section \ref{we}. We can use Theorems \ref{weyl} and \ref{ort1} to easily compute the norms of polynomials $p_\lambda(x; q,t)$. Namely, take $\wla=\wrho+\lambda$ with $\lambda\in P_{++}$,
and consider the function $\Phi_-(\wla, x)$ as defined in \eqref{phi}. Then we can use Theorem \ref{ortrel} to compute the integral
\begin{equation*}
  \int_{C_\xi/\kappa Q^\vee}\frac{\Phi_-(\wla,x)\Phi_-(\wla, -x)}{\Delta(x)\Delta(-x)}\,dx\,.
\end{equation*}
Indeed, expanding $\Phi_-$ in terms of $\psi$'s and using the fact that $w\wla=w'\wla$ only when $w=w'$, we obtain that the integral equals
\begin{equation*}
\sum_{w\in W}(-1)^M\dst(w\wla)\dst(-w\wla)=  |W|(-1)^{M}\dst(\wla)\dst(-\wla)\,.
\end{equation*}
(Here we used the $W$-invariance of $\dst(\lambda)\dst(-\lambda)$.)

According to \eqref{weyl2}, we have
\begin{equation}\label{weyl3}
\Phi_-(\wla, x)=(-1)^M\dst(\wla) \delta(x)\,
p_{\lambda}(x)\,,\qquad p_\lambda(x)=p_\lambda(x;q,q^{m+1})\,.
\end{equation}
Substituting this into the integral gives:
\begin{equation*}
\int_{C_\xi/\kappa Q^\vee}p_\lambda(x)p_\lambda(-x)\frac{\delta(x)\delta(-x)}{\Delta(x)\Delta(-x)}\,dx = |W|(-1)^{M}\frac{\dst(-\wla)}{\dst(\wla)}\,.
\end{equation*}
Now it is easy to check that
\begin{equation}\label{dena}
\frac{\delta(x)\delta(-x)}{\Delta(x)\Delta(-x)}=C^{-1}(-1)^{|R_+|}\nabla(x;q,q^{m+1})\,,
\end{equation}
where $\delta$ is as in \ref{we} and $C$ is the constant \eqref{cc}--\eqref{ccc}.

As a result, we obtain that
\begin{equation}\label{tutu}
\int_{C_\xi/\kappa Q^\vee}p_\lambda(x)p_\lambda(-x)\nabla(x)\,dx\\=C(-1)^{\wmm}|W|
\frac{\dst(-\lambda-\wrho)}{\dst(\lambda+\wrho)}\,,
\end{equation}
where we used $\wmm:=\sum_{\alpha\in R_+} (m_\alpha+1)$.

Since now the integrand has no poles, we can shift the cycle $C_\xi$ back to $i\V$, so the left-hand side becomes the Macdonald scalar product $\langle p_\lambda, p_\lambda \rangle$.
This leads to the formula for the norms of $p_\lambda(x; q, t)$ in the case $t=q^{m+1}$, cf. \cite{Ch1,M4}.
\hfill$\Box$

\begin{remark}
Note that the above proof of the norm identity does not use shift operators or an inductive step from $m$ to $m+1$. In that respect it is very different from other known proofs that use the idea going back to \cite{O}. There is also an alternative method of deriving the formula for $\langle p_\lambda, p_\lambda \rangle / \langle 1, 1 \rangle$, using intertwiners, see \cite{Ch5,Ch6}. But then one still has a problem of computing the so-called constant term $\langle 1, 1 \rangle$.
\end{remark}

\begin{remark}\label{macpex}
The above argument can, in fact, be used to give a simpler proof of \eqref{weyl2} together with the existence of Macdonald polynomials.
Namely, let us {\it define} $p_\lambda$ in terms of $\psi$ with the help of formula \eqref{weyl3}. Then, just by using the quasi-invariance and skew-symmetry of $\Phi_-$ (like at the first step in the proof of Theorem \ref{weyl}), we conclcude that such $p_\lambda$ will be a symmetric polynomial of the form \eqref{macp}. It remains to show that thus defined functions satisfy \eqref{mort} for $t=q^{m+1}$. To this end, we know that for dominant weights $\lambda\ne \mu$ we have by Theorem \ref{ort1} that
\begin{equation*}
  \int_{C_\xi/\kappa Q^\vee}\frac{\Phi_-(\wla,x)\Phi_-(\wmu, -x)}{\Delta(x)\Delta(-x)}\,dx=0\,.
\end{equation*}
This gives that
\begin{equation*}
\int_{C_\xi/\kappa Q^\vee}p_\lambda(x)p_\mu(-x)\nabla(x)\,dx\\=0\,,
\end{equation*}
similarly to the way we obtained \eqref{tutu} above. Since now the integrand has no poles, we can shift the cycle $C_\xi$ back to $i\V$, so this relation turns into $\langle p_\lambda, p_\mu \rangle=0$. Thus, we showed that \eqref{weyl3} holds true for some $p_\lambda\in\a^W$ which will satisfy \eqref{macp}--\eqref{mort}. This simultaneously proves the existence of $p_\lambda$ and the relation \eqref{weyl3}.
\end{remark}

\subsection{The case of $|q|=1$}

The relations \eqref{ortrel} and their proof remain true for $q\in\c^\times$ with $|q|\ne 1$. In that case one still uses $C_\xi=\xi+\kappa\V$ with $\kappa$ given by \eqref{kappa}. Moreover, a similar result is true for
$|q|=1$ when $\kappa\in \R$. In that case we know that the BA function $\psi$ exists and is analytic in $q$
provided \eqref{roots}--\eqref{roots2}. Then we have the following analogue of Theorem \ref{ort1}.

\begin{theorem}\label{ort2} Assume that $|q|=1$ and conditions \eqref{roots}--\eqref{roots2} are satisfied. Put $C_\xi=i\xi+\V$ with $\xi\in\V$, assuming $\xi$ is regular, i.e. $\langle\alpha,\xi\rangle\ne 0$ for all $\alpha\in R$. Then for such $C_\xi$ and $\lambda, \mu\in \V$ with $\lambda-\mu\in P$, the relations \eqref{ortrel} remain valid.
\end{theorem}
For generic $q$ on the unit circle this is proved similarly to Theorem \ref{ort1}. Namely, due to a cancelation of residues the integral does not depend on $\xi$ (provided it stays regular), after which the integral is evaluated by letting $\xi\to\infty$. For non-generic $q$ such that \eqref{roots}--\eqref{roots2} are satisfied, the integrand depends analytically on $q$, so the result survives when $q$ approaches those values. \hfill$\Box$

\section{Cherednik--Macdonald--Mehta integral}\label{mmintegral}

Throughout this section $0<q<1$ and $\psi(\lambda,x)$ is the normalized BA function of type {\bf b} associated to $(R,m)$. Recall that in this case we have $(\rst,\mst)=(R,m)$, so $\psi(\lambda,x)=\psi(x,\lambda)$ and $\dst=\Delta$, where $\Delta$ is given by \eqref{Deltam} with $q_\alpha=q^{\langle\alpha,\alpha\rangle/2}$.

Let $dx$ be the translation invariant measure on $C_\xi=\xi+i\V$, normalized by the condition
\begin{equation*}
    \int_{C_\xi} q^{-|x|^2/2} dx = 1\,,\qquad |x|^2:=\langle x,x\rangle\,.
\end{equation*}
(Note that $|x|^2<0$ for $x\in i\V$.)

Our goal is to prove the following integral identity (its further generalizations, including cases {\bf a} and {\bf c} are discussed in Section \ref{twgauss} of the Appendix).
\begin{theorem}\label{mehta}  For any $\lambda, \mu\in \v$ and any big $\xi\in\V$ we have
\begin{equation}\label{mehtaint}
\int_{C_\xi}\frac{\psi(\lambda,x)\psi(\mu, x)}{\Delta(x)\Delta(-x)}q^{-|x|^2/2}\,dx=(-1)^MC^{-1/2} q^{(|\lambda|^2 + |\mu|^2)/2} \psi(\lambda,\mu)\,,
\end{equation}
where $C$ is the constant \eqref{cc} and $M=\sum_{\alpha\in R_+}m_\alpha$.
\end{theorem}

The proof of the theorem will be based on the following proposition, similar to Proposition \ref{res}.

\begin{prop}\label{res1}
Let $I(\xi)$ denote the integral in the left-hand side of
\eqref{mehtaint}. Then $I(\xi)$ does not depend on $\xi$ provided
$\xi$ remains big in the sense of \eqref{big}.
\end{prop}

Note that in this case we integrate over a non-compact cycle, but the integral converges absolutely due to the rapidly decaying factor $q^{-|x|^2/2}$. The proposition can be proved by looking at the residues of the integrand in \eqref{mehta} given by
\begin{equation*}\label{integrand1}
G(x)= \frac{\psi(\lambda,x)\psi(\mu,x)}{\Delta(x)\Delta(-x)}q^{-|x|^2/2}\,.
\end{equation*}
Without the factor $q^{-|x|^2/2}$ we would have a
cancelation of the residues as in Lemma \ref{cancel}. Now, the
crucial fact is that the function $g(x)=q^{-|x|^2/2}$ satisfies the quasi-invariance conditions \eqref{quasi}. Indeed, we have for $j\in\Z$ that
\begin{equation}\label{gain}
g(x-\frac12 j\alpha)/g(x+\frac12 j\alpha)=q^{j\langle\alpha,x\rangle}=1\qquad\text{for}\ q^{\langle\alpha,x\rangle}=1\,.
\end{equation}
As a result, the same
cancelation of the residues as in Lemma \ref{cancel} also takes
place for $G$, and the rest of the proof remains the same.
\hfill$\Box$

\medskip
Before proving the theorem, let us mention a `compact' version of the integral \eqref{mehtaint}. Let $\theta(x)$ denote the theta-function associated with the lattice $P$:
\begin{equation}\label{theta}
\theta(x)=\sum_{\gamma\in P} q^{\langle\gamma,x\rangle}q^{|\gamma|^2/2}\,.
\end{equation}
We have the following standard fact (see e.g. \cite[Lemma 4.3]{EV}):
\begin{lemma}
If $f(x)$ is a smooth function on $C_\xi$, which is periodic with respect to the lattice $\kappa Q^\vee$, then
\begin{equation*}
 \int_{C_\xi} f(x)q^{-|x|^2/2}\,dx =  \int_{C_\xi/\kappa Q^\vee} f(x)\theta(x)\,dx\,.
\end{equation*}
\end{lemma}\hfill$\Box$

\medskip

When $\lambda+\mu\in P$, the product $\psi(\lambda,x)\psi(\mu,x)$ is $\kappa Q^\vee$-periodic. In that case we can reformulate Theorem \ref{mehta} in the following way.

\begin{theorem}\label{mehtac}  If $\xi\in\V$ is big and $\lambda+\mu\in P$, then
\begin{equation*}\label{mehtaintc}
\int_{C_\xi/\kappa Q^\vee}\frac{\psi(\lambda,x)\psi(\mu, x)}{\Delta(x)\Delta(-x)}\,\theta(x)\,dx=(-1)^MC^{1/2} q^{(|\lambda|^2+|\mu|^2)/2} \psi(\lambda,\mu)\,.
\end{equation*}
\end{theorem}

\subsection{Proof of Theorem \ref{mehta}}\label{expansion}

Let us first assume that $\xi$ belongs to the negative Weyl chamber,
i.e. $\langle \alpha, \xi\rangle \ll 0$ for $\alpha\in R_+$. The
denominator in \eqref{integrand1} can be presented as
\begin{equation*}
\Delta(x)\Delta(-x)= (-1)^M q^{2\langle \rho, x\rangle}\prod_{\alpha\in R_+}
\prod_{j=\pm 1}^{\pm m_\alpha}\left(1-q_\alpha^{j}q^{-\langle\alpha, x\rangle}\right)\,.
\end{equation*}
For $x\in \xi+i\V$ we have $\mathrm{Re}\langle\alpha, x\rangle =
\langle\alpha, \xi\rangle\ll 0$ and $\left|q^{-\langle\alpha, x\rangle}\right|\ll 1$ for $\alpha\in R_+$. Therefore, we can expand each of the factors $(1-q_\alpha^{j}q^{-\langle\alpha,
x\rangle})^{-1}$ into a geometric series and obtain that
\begin{equation}\label{exp1}
\left[\Delta(x)\Delta(-x)\right]^{-1}=q^{-2\langle \rho,
x\rangle}\sum_{\gamma\in Q_-}a_\gamma q^{\langle \gamma,
x\rangle}\,,\quad a_0=(-1)^M\,.
\end{equation}
The series converges uniformly and absolutely on $C_\xi$ provided
that $\xi$ lies deep inside the negative Weyl chamber. Using \eqref{exp1} and \eqref{expand}, we can expand the function
\begin{equation*}
F(x)= \frac{\psi(\lambda,x)\psi(\mu,x)}{\Delta(x)\Delta(-x)}
\end{equation*}
into a similar convergent series:
\begin{equation}\label{exp2}
F(x)=q^{\langle\lambda+\mu,x\rangle}\sum_{\gamma\in P_-}f_\gamma q^{\langle\gamma,x\rangle}\,,\qquad  f_0=(-1)^M\Delta(\lambda)\Delta(\mu)\,.
\end{equation}

All the coefficients $f_\gamma$ in the series will be functions of $\lambda$ and $\mu$ of the form:
\begin{equation}\label{fgamma}
f_\gamma=\sum_{\nu, \nu'\in\mathcal N\cap\, \rho+P}
a_{\gamma;\nu,\nu'}q^{\langle\nu,\lambda\rangle}q^{\langle\nu',\mu\rangle}
\end{equation}
with suitable coefficients $a_{\gamma; \nu,\nu'}$ (this is immediate from \eqref{psi}).

Note that the coefficients $a_\gamma$ in
\eqref{exp1} and, as a consequence, $a_{\gamma; \nu,\nu'}$  in \eqref{fgamma} have
moderate (`exponentially linear') growth, namely,
\begin{equation}\label{growth}
    |a_\gamma|<Aq^{\langle u, \gamma\rangle}\quad\text{and}\quad
    |a_{\gamma; \nu,\nu'}|<A'q^{\langle u', \gamma\rangle}\quad\text{for all}\ \gamma, \nu, \nu'\,,
\end{equation}
for suitable constants $A, A'$ and vectors $u,u'\in\V$.

Substituting the series \eqref{exp2} into \eqref{mehtaint} and
integrating termwise, we obtain a series expansion for the integral
\eqref{mehtaint} as follows:
\begin{multline}\label{expint}
I(\xi)=\sum_{\gamma\in P_-}f_{\gamma}\int_{C_\xi} q^{\langle
\lambda+\mu+\gamma, x\rangle}q^{-|x|^2/2}\,dx =\\
\sum_{\gamma\in P_-}f_\gamma q^{|\lambda+\mu+\gamma|^2/2}=q^{|\lambda+\mu|^2/2}\sum_{\gamma\in
P_-}f_\gamma q^{\langle \gamma,
\lambda+\mu\rangle}q^{|\gamma|^2/2}\,.
\end{multline}
Let us view now this expression as a function of $\lambda$. Since each of the coefficients $f_\gamma$, as a function of $\lambda$, is a polynomial in $\A$ whose exponents spread over the polytope $\mathcal N$, we have that
\begin{equation}\label{explambda}
    I(\xi)=q^{|\lambda+\mu|^2/2}\sum_{\gamma\in \rho + P_-}g_\gamma q^{\langle \gamma,
\lambda\rangle}\,,
\end{equation}
with some coefficients $g_\gamma$ that depend on $\mu$. It follows from \eqref{exp2} that
\begin{equation}\label{gamma1}
g_{\rho}=(-1)^M \Delta(\mu) \prod_{\alpha\in R_+}q_\alpha^{-m_\alpha(m_\alpha+1)/4}\,.
\end{equation}
From the way the expression \eqref{explambda} was obtained, it is clear that each $g_\gamma$ is a finite combination of the terms $a_{\gamma';\nu,\nu'}q^{\langle \gamma'+\nu', \mu\rangle}q^{|\gamma'|^2/2}$ with $\gamma'\in\gamma+\mathcal N$. Since we are keeping $\mu$ fixed, we can use \eqref{growth} to obtain an estimate for $g_\gamma$:
\begin{equation}\label{growth1}
  |g_\gamma| < B q^{\langle v, \gamma\rangle}q^{|\gamma|^2/2}\quad\text{for all}\ \gamma\,,
\end{equation}
with a suitable constant $B$ and $v\in\V$.

It follows that the coefficients $g_\gamma$ are fast decreasing as $|\gamma|\to\infty$, therefore, the series \eqref{explambda} defines an analytic function of $\lambda$ of the form
\begin{equation}\label{explambda1}
    I(\xi)=q^{|\lambda+\mu|^2/2}\sum_{\gamma\in P}g_\gamma q^{\langle \gamma,
\lambda\rangle}\,,
\end{equation}
where $g_\gamma=0$ unless $\gamma\in \rho + P_-$.
Note that presentation of $I(\xi)$ in the form \eqref{explambda1} is unique, as it comes from the Fourier series of $g(\lambda)=I(\xi)q^{-|\lambda+\mu|^2/2}$ on the torus $T=i\V/\kappa Q^\vee$.

We arrive at the conclusion that for $\xi$ deep in the negative Weyl chamber, $I(\xi)$ is given by the series \eqref{explambda1}, where $g_\gamma=0$ unless $\gamma\in\rho+P_-$. If we apply the same arguments for, say, $\xi'$ in the positive Weyl chamber, we would get a similar series for $I(\xi')$, but with nonzero Fourier coefficients only for $\gamma\in -\rho+P_+$. Since $I(\xi)=I(\xi')$, we conclude that the two series coincide and, therefore, have only a finite number of terms. Moreover, by moving $\xi$ to various Weyl chambers, we conclude that all $\gamma$ with $g_\gamma\ne 0$ must lie within the polytope with vertices $\{w\rho\,|\, w\in W\}$, i.e. the polytope $\nn$. Therefore,
\begin{equation*}\label{explambda2}
    I(\xi)=q^{|\lambda+\mu|^2/2}\sum_{\gamma\in \nn\cap\,\rho+P}g_\gamma q^{\langle \gamma,
\lambda\rangle}\,.
\end{equation*}

As a function of $\lambda$, $I(\xi)$ inherits from $\psi(\lambda,x)$ the properties \eqref{axxpsi}. The multiplication by $q^{|\lambda|^2/2}$ does not affect these properties (see \eqref{gain}). Thus, the function $I(\xi)q^{-(|\lambda|^2+|\mu|^2)/2}$ satisfies \eqref{psi} and \eqref{axpsi} (with $(\lambda, \mu)$ taking place of $(x, \lambda)$). By Theorem \ref{bau}, these properties characterize $\psi$ uniquely up to a factor depending on the second variable. Hence,
\begin{equation*}
I(\xi)q^{-(|\lambda|^2+|\mu|^2)/2}=C(\mu)\psi(\lambda,\mu)\,,\quad\text{for some}\ C(\mu)\,.
\end{equation*}
Comparing \eqref{norm0} and \eqref{gamma1}, we conclude that
$C(\mu)=(-1)^MC^{-1/2}$, as needed. This finishes the proof of the theorem. \hfill$\Box$

\subsection{Integral transforms}

In this section $\psi(\lambda,x)$ is the normalized BA function in any of the cases {\bf a}, {\bf b} or {\bf c}.
Let us introduce
\begin{equation}\label{F}
F(\lambda, x)=\frac{\psi(\lambda, -x)}{\dst(\lambda)\Delta(x)}\,.
\end{equation}
Note that $F(x,\lambda)=\fst(\lambda,x)$, where $\fst$ is the counterpart of $F$ for the dual data $(\rst,\mst)$. In particular, in case {\bf b} we have
\[
F(\lambda,x)=\frac{\psi(\lambda, -x)}{\Delta(\lambda)\Delta(x)}\,,\qquad F(\lambda,x)=F(x,\lambda)\,.
\]

The relations \eqref{ortrel} can be rewritten as
\begin{gather}\nonumber
 \int_{C_\xi/\kappa Q^\vee}F(\lambda,-x)F(\mu,x)\,dx=\delta_{\lambda, \mu}Q^{-1}(\lambda)\,,
\qquad\text{where}\\
\label{q}
Q(\lambda)=(-1)^{M}\frac{\dst(\lambda)}{\dst(-\lambda)}\quad\text{and}\quad\lambda-\mu\in P
\,.
\end{gather}
This makes them look similar to \cite[Theorem 2.2]{EV}.

The formula \eqref{mehtaint} in case {\bf b}, when written in terms of $F(\lambda,x)$, is equivalent to
\begin{equation}\label{mint}
\int_{C_\xi}F(\lambda, -x)F(\mu,x)q^{-|x|^2/2}\,dx=(-1)^MC^{-1/2} q^{(|\lambda|^2+|\mu|^2)/2} F(\lambda,\mu)\,,
\end{equation}
where $C$ is the constant \eqref{cc} (cf. \cite[Theorem 2.3]{EV}).

We can use functions \eqref{F} to define Fourier transforms, following the approach of \cite{EV}. Since the proofs repeat verbatim those in \cite{EV}, we will only formulate the results, referring the reader to the above paper for the details.

For $\xi, \eta\in\V$ consider the imaginary subspace $C_\xi=\xi+i\V$ and the real subspace $D_\eta=i\eta+\V$.
Let $\S(C_\xi)$ and $\S(D_\eta)$ be the Schwartz spaces of functions on $C_\xi$ and $D_\eta$ respectively. Introduce the spaces $\S_\eta(C_\xi)=\{\phi\,:\ C_\xi\to\c\,|\,q^{2i\langle\eta, x\rangle}\phi(x)\in \S(C_\xi)\}$ and $\S_\xi(D_\eta)=\{\phi\,:\ D_\eta\to\c\,|\,q^{-2\langle\xi, \lambda\rangle}\phi(\lambda)\in \S(D_\eta)\}$. Obviously, these spaces are canonically isomorphic to $\S(C_\xi)$ and $\S(D_\eta)$. The modified Fourier transform  $f(x)\mapsto \hat f(\lambda):=\int_{C_\xi} q^{2\langle \lambda, x\rangle}f(x)\,dx$ defines
an isomorphism $\S_\eta(C_\xi) \to \S_\xi(D_\eta)$. The inverse transform $\hat f(\lambda)\mapsto f(x)$ is given
by the formula $f(x)=\int_{D_\eta} q^{-2\langle \lambda, x\rangle}\hat f(\lambda)\,d\lambda$. This fixes uniquely a normalization of the Lebesgue measure $d\lambda$ on $D_\eta$, which will be used from now on.

Consider two integral transformations
\begin{equation*}\label{tim}
K_{\mathrm{Im}}\,:\ \S_\eta(C_\xi) \to \S_\xi(D_\eta)\,,\quad f(x)\mapsto \int_{C_\xi} F(\lambda,-x)f(x)\,dx\,,
\end{equation*}
and
\begin{equation*}\label{tre}
K_{\mathrm{Re}}\,:\ \S_\xi(D_\eta) \to \S_\eta(C_\xi)\,,\quad f(\lambda)\mapsto \int_{D_\eta} F(\lambda, x)Q(\lambda)f(\lambda)\,d\lambda\,,
\end{equation*}
where $Q$ is given in \eqref{q}.

\begin{theorem}[cf. {\cite[Theorem 2.4]{EV}}]\label{trans}
Assume that $\xi\in\V$ is big and $\eta\in\V$ is regular in a sense that $\dst(\lambda)\dst(-\lambda)$ is non-vanishing on $D_\eta$. Then the integral transforms are well defined, continuous in the Schwartz topology, and are inverse to each other,
\begin{equation*}
K_{\mathrm{Im}}\,K_{\mathrm{Re}}= Id\,,\quad  K_{\mathrm{Re}}\,K_{\mathrm{Im}}= Id\,.
\end{equation*}
\end{theorem}

\subsection{Cherednik--Macdonald--Mehta integral over real cycle}

In case {\bf b}, we can use Theorem \ref{trans} to derive a `real' counterpart of Theorem \ref{mehta}, similarly to  \cite{EV}. Namely, formula \eqref{mint} says that for a fixed generic $\mu$ one has
\begin{equation*}
 K_{\mathrm{Im}}\left( F(\mu,x)q^{-(|x|^2+|\mu|^2)/2} \right) = (-1)^MC^{-1/2} q^{|\lambda|^2/2} F(\lambda,\mu)\,.
\end{equation*}
Applying $K_{\mathrm{Re}}$ to both sides, we obtain
\begin{equation*}
 F(\mu,x)q^{-(|x|^2+|\mu|^2)/2}=(-1)^M C^{-1/2} \int_{D_\eta} F(\lambda, x)F(\lambda,\mu)Q(\lambda)q^{|\lambda|^2/2}\,d\lambda\,.
\end{equation*}
Expressing everything back in terms of $\psi$, we obtain
\begin{equation*}
\int_{D_\eta} \frac{\psi(\lambda,-x)\psi(\lambda,-\mu)}{\Delta(\lambda)\Delta(-\lambda)}q^{|\lambda|^2/2}\,d\lambda
= C^{1/2}\psi(\mu,-x)q^{-(|x|^2+|\mu|^2)/2}\,.
\end{equation*}
In the derivation of this formula we assumed that $x\in C_\xi$ and $\mu$ is generic. However, since both sides are obviously analytic in $\mu$ and $x$, the formula remains valid for all $\mu, x\in \v$. After rearranging and using that $\psi(\lambda,x)=\psi(x,\lambda)$, we get the following result.

\begin{theorem}[cf. {\cite[Theorem 2.6]{EV}}]\label{remehta} Let $D_\eta=i\eta+\V$ with $\eta$ regular in the sense of Theorem \ref{trans}. Then for any $\mu,\nu\in\v$ we have
\begin{equation*}\label{remehtaint}
 \int_{D_\eta} \frac{\psi(\mu,\lambda)\psi(\nu,-\lambda)}{\Delta(\lambda)\Delta(-\lambda)}\,q^{|\lambda|^2/2}\,d\lambda = C^{1/2} q^{-(|\nu|^2+|\mu|^2)/2}\psi(\mu,\nu)\,,
\end{equation*}
where $C$ is the constant \eqref{cc}.
\end{theorem}

\subsection{Symmetric version}

Similarly to Section \ref{norms}, we can use the generalized Weyl formula to derive the analogues of Theorems \ref{mehtaint}, \ref{mehtaintc}, \ref{remehtaint} for Macdonald polynomials $p_\lambda$ in case {\bf b}. This gives a simple proof of the identities proved by Cherednik in \cite{Ch3} using the double affine Hecke algebras.

Let $p_\lambda$ and $\nabla$ denote the Macdonald polynomials and weight function, respectively, in case {\bf b} with $t=q^{m+1}$. For $\lambda, \mu\in P_{++}$ let us put $\wla=\lambda+\wrho$, $\wmu=\mu+\wrho$ in the notations of Sections \ref{we} and \ref{norms}. Also, put
\begin{equation}\label{wdel}
\wdel(x):=C\delta(x)/\Delta(x)=C\Delta(-x)\delta_0(x)\,,
\end{equation}
where $C$ is the constant \eqref{cc}.

\begin{theorem}[cf. {\cite[Theorems 1.1 and 1.2]{Ch3}}]\label{mm}  We have the following identities:
\begin{gather*}\label{mm1}
\int_{i\V}p_\lambda(x)p_\mu(x)q^{-|x|^2/2}\nabla(x)\,dx=
(-1)^MC^{-1/2}\,|W|q^{(|\wla|^2+|\wmu|^2)/2}\wdel(\wmu)p_\lambda(\wmu)\,,\\\label{mm2}
\int_{i\V/\kappa Q^\vee}p_\lambda(x)p_\mu(x)\theta(x)\nabla(x)\,dx=(-1)^MC^{-1/2}\,|W|q^{(|\wla|^2+|\wmu|^2)/2}
\wdel(\wmu)p_\lambda(\wmu)\,,
\\\label{mm3}
\int_{\V}p_\lambda(x)p_\mu(-x)q^{|x|^2/2}\nabla(x)\,dx=C^{1/2}|W|q^{-(|\wla|^2+|\wmu|^2)/2}
\wdel(\wmu)p_\lambda(\wmu)\,.
\end{gather*}
Here $C$ is the constant \eqref{cc} and $\theta(x)$ is the theta-function \eqref{theta}.

\end{theorem}

\begin{proof}
The first two formulas are obviously equivalent. We will only derive the first identity, since the third one is entirely similar.

Consider the integral
\begin{equation*}
\int_{C_\xi}\frac{\Phi_-(\wla,x)\Phi_-(\wmu, x)}{\Delta(x)\Delta(-x)}q^{-|x|^2/2}\,dx\,,
\end{equation*}
where $\Phi_-(\wla,x)$, $\Phi_-(\wmu, x)$ are as in \eqref{weyl2}. Expanding $\Phi_-$ in terms of $\psi$ and applying formula \eqref{mehtaint}, we conclude that the integral equals
\begin{equation*}
(-1)^MC^{-1/2}q^{(|\wla|^2+|\wmu|^2)/2}\sum_{w,w'\in W}(-1)^{ww'}\psi(w\wla,w'\wmu)\,.
\end{equation*}
Using Lemma \ref{w}$(i)$, we get that
\begin{equation*}
\sum_{w,w'\in W}(-1)^{ww'}\psi(w\wla,w'\wmu)=|W|\sum_{w\in W}(-1)^w\psi(w\wla, \wmu)=|W|\Phi_-(\wla,\wmu)\,.
\end{equation*}
Therefore,
\begin{equation*}
\int_{C_\xi}\frac{\Phi_-(\wla,x)\Phi_-(\wmu, x)}{\Delta(x)\Delta(-x)}q^{-|x|^2/2}\,dx=(-1)^MC^{-1/2}|W|q^{(|\wla|^2+|\wmu|^2)/2}
\Phi_-(\wla,\wmu)\,.
\end{equation*}
After substituting expression \eqref{weyl2} for $\Phi_-$ and rearranging, we get
\begin{multline*}
\int_{C_\xi}p_\lambda(x)p_\mu(x)\frac{\delta(x)\delta(x)}{\Delta(x)\Delta(-x)}q^{-|x|^2/2}\,dx\\=
(-1)^MC^{-1/2}|W|q^{(|\wla|^2+|\wmu|^2)/2}\frac{\delta(\wmu)}{\Delta(\wmu)}p_\lambda(\wmu)\,.
\end{multline*}
It follows from \eqref{delta}, \eqref{dena} that
\begin{equation*}
\frac{\delta(x)\delta(x)}{\Delta(x)\Delta(-x)}=C^{-1}\nabla(x)\,.
\end{equation*}
As a result, we obtain that
\begin{equation*}
\int_{C_\xi}p_\lambda(x)p_\mu(x)q^{-|x|^2/2}\nabla(x)\,dx=
(-1)^MC^{-1/2}|W|q^{(|\wla|^2+|\wmu|^2)/2}\wdel(\wmu)p_\lambda(\wmu)\,.
\end{equation*}
Since the integrand in the left-hand side is non-singular, we can shift the contour back to $i\V$, and this leads to the required result.
\end{proof}

\subsection{$q$-Macdonald--Mehta integral}
Putting $\lambda=\mu=0$ in Theorem \ref{mm} gives us different variants of the $q$-analogue of the Macdonald--Mehta integral \cite{M3}, due to Cherednik \cite{Ch3}. For instance, we have
\begin{equation}\label{mmi}
\int_{i\V}q^{-|x|^2/2}\nabla(x;q,q^{m+1})\,dx=(-1)^MC^{1/2}|W|q^{|\wrho|^2}
\wdel(\wrho)\,.
\end{equation}
If we denote $k:=m+1$ and $\rho_k:=\frac12\sum_{\alpha\in R_+}k_\alpha\alpha$, then \eqref{mmi} can be written as
\begin{equation}\label{rhs}
\int_{i\V}q^{-|x|^2/2}\nabla(x;q,q^{k})\,dx=|W|\prod_{\alpha\in R_+}\frac{(q^{\langle\alpha,\rho_k\rangle};q_\alpha)_\infty}
{(q_\alpha^{k_\alpha}q^{\langle\alpha,\rho_k\rangle};q_\alpha)_\infty}\,.
\end{equation}
This makes it equivalent to the $q$-Macdonald--Mehta integral from \cite{Ch3}. Each quantity $q^{\langle\beta,\rho_k\rangle}$ with $\beta\in R_+$ can be expressed as a polynomial in $t_\alpha=q_\alpha^{k_\alpha}$, after which the right-hand side of \eqref{rhs} allows analytic continuation to all complex values of $t_\alpha$. According to \cite{Ch3}, \eqref{rhs} remains true for any $k_\alpha>0$. This, however, does not allow  $t=q^{-m}$ with $m_\alpha\in\Z_+$, so it is not clear from the results of \cite{Ch3} how to extend the formula \eqref{rhs} to such values.

On the other hand, Theorem \ref{mehta} allows us to evaluate directly an integral of Macdonald--Mehta type for $t=q^{-m}$. Namely, let us put $\lambda=\mu=\rho$ in \eqref{mehtaint}. By Proposition \ref{eval}, $\psi$ in that case becomes a nonzero constant $\Delta(-\rho)$. We therefore obtain
\begin{equation}\label{mmev1}
 \int_{C_\xi}\left(\Delta(x)\Delta(-x)\right)^{-1}q^{-|x|^2/2}\,dx=(-1)^MC^{-1/2}q^{|\rho|^2}\Delta^{-1}(-\rho)\,.
\end{equation}
Here $-\rho=\rho_{-m}$ in the above notation. This identity can be written as
\begin{equation}\label{mmev}
\int_{\xi+i\V}q^{-|x|^2/2}\nabla(x;q,q^{-m})\,dx=\prod_{\alpha\in R_+}\frac{(q^{\langle\alpha,\rho_{-m}\rangle};q_\alpha)_\infty}
{(q_\alpha^{-m_\alpha}q^{\langle\alpha,\rho_{-m}\rangle};q_\alpha)_\infty}\,,
\end{equation}
where the expression in the right-hand side is to be taken formally:
\[
\prod_{\alpha\in R_+}\frac{(q^{\langle\alpha,\rho_{-m}\rangle};q_\alpha)_\infty}
{(q_\alpha^{-m_\alpha}q^{\langle\alpha,\rho_{-m}\rangle};q_\alpha)_\infty}=\prod_{\alpha\in R_+}\prod_{j=1}^{m_\alpha}
\left(1-q_\alpha^{-j}q^{\langle\alpha,\rho_{-m}\rangle}\right)^{-1}\,.
\]
One can check that this expression coincides with the right-hand side of \eqref{rhs} evaluated at $k_\alpha=-m_\alpha\in\Z_-$, cf. Remark \ref{evalcompare}. (This is not entirely trivial, cf. \cite{M1} where expressions similar to \eqref{rhs} are evaluated at $k_\alpha=0$.) Thus, \eqref{mmev} can be viewed as an analytic continuation of \eqref{rhs}, which justifies $C_\xi$ being a correct contour in the case $t=q^{-m}$.

\medskip

\begin{remark}
An alternative approach would be to keep the same contour, but add corrections by taking into account the residues of the integrand between $i\V$ and $C_\xi$. This looks more complicated but has an advantage of handling the case of $t=q^{-m}$ with non-integer $m$. The results of \cite{KS} seem to indicate such a possibility (at least, in rank one), see also \cite{Ch4}. On the other hand, we note that in Theorems \ref{ort2} and \ref{remehta} the integration is performed over a real cycle which does not depend of $m$. Therefore, we expect these statements to remain valid (by analytic continuation in $m$) for non-integer $m$, with a suitably defined $\psi(\lambda,x)$. The same remark applies to the summation formula \eqref{gsum} below.
\end{remark}

\medskip

\begin{remark}
BA functions can be also defined and constructed in the rational and trigonometric settings, see \cite{CFV99,C00}. They can be viewed as suitable limits of $\psi(\lambda,x)$ when $q\to 1$, so some of the above results survive in such a limit. For example, the orthogonality relations can be stated and proved in a similar fashion. Also, the Cherdnik--Macdonald--Mehta integral survives in the rational (but not trigonometric) limit. Note that in the rational case $\psi(\lambda,x)$ exists also in non-crystallographic cases (for instance, for the dihedral groups). However, our proof of \eqref{mehtaint} does not work in the rational case, so by allowing $q\to 1$ we can only obtain the result for the Weyl groups. It would be therefore interesting to find a direct proof of Cherednik--Macdonald--Mehta integral for BA functions in the rational setting, cf. \cite{E} where the Macdonald--Mehta--Opdam integral is computed for all Coxeter groups in a uniform fashion.
\end{remark}

\newpage

\centerline{\bf Appendix}
\medskip
\centerline{\bf by Oleg Chalykh}

\bigskip

\section{Summation formulas}

In \cite{Ch3} Cherednik gives a version of Theorem \ref{mm} with integration replaced by summation. Here we prove a similar result for BA functions, which leads to new identities of Cherednik type. This also gives an elementary proof of Cherednik's results \cite[Theorem 1.3]{Ch3}.

We will consider case {\bf b}, so $(R,m)$ is a reduced irreducible root system with $W$-invariant multiplicities $m_\alpha\in \Z_+$, and $(\rst, \mst)=(R,m)$. Generalizations to cases {\bf a} and {\bf c} are considered in Section \ref{ac}.
Throughout this section $|q|<1$.

For any $f(x)$ and $\xi\in\v$, define $\langle f \rangle_\xi$ as
\begin{equation*}\label{ji}
\langle f \rangle_\xi=\sum_{\gamma\in P} f(\xi+\gamma)\,,
\end{equation*}
assuming convergence. For instance,
\[
\langle q^{|x|^2/2}\rangle_\xi=q^{|\xi|^2/2}\theta(\xi)\,,
\]
where $\theta(x)$ is the theta function \eqref{theta}.

\begin{theorem}[cf. {\cite[Theorem 1.3]{Ch3}}]\label{jg} For any $\lambda,\mu\in \V$ and $\xi\in\v$ we have
\begin{equation}\label{gsum}
\left\langle\frac{\psi(\lambda,x)\psi(\mu,-x)}{\Delta(x)\Delta(-x)}q^{|x|^2/2}\right\rangle_\xi=
C^{1/2}q^{-\frac{|\lambda|^2+|\mu|^2}{2}}\psi(\lambda,\mu)\langle q^{\frac{|x+\lambda-\mu|^2}{2}}\rangle_\xi\,.
\end{equation}
where $C$ is the constant \eqref{cc}. In particular, for $\lambda-\mu\in P$ we get
\begin{equation}\label{gsum1}
\left\langle\frac{\psi(\lambda,x)\psi(\mu,-x)}{\Delta(x)\Delta(-x)}q^{|x|^2/2}\right\rangle_\xi=
C^{1/2}q^{-\frac{|\lambda|^2+|\mu|^2}{2}}\psi(\lambda,\mu)q^{|\xi|^2/2}\theta(\xi)\,.
\end{equation}
We assume that $\xi$ is generic so that the left-hand side of \eqref{gsum}, \eqref{gsum1} is well-defined.
\end{theorem}
\begin{proof}
Denote
\begin{equation}\label{f}
F(\lambda,\mu;x)=\frac{\psi(\lambda,x)\psi(\mu,-x)}{w(x)}q^{|x|^2/2}\,,\qquad w(x):=\Delta(x)\Delta(-x)\,.
\end{equation}
Using \eqref{psi}, \eqref{expand} and duality, one easily checks
that for every $v\in Q^\vee$ we have
\begin{gather}
\label{tp1}
F(\lambda,\mu; x+\kappa v)= e^{2\pi i\langle x+\lambda-\mu,v\rangle}e^{\pi i\kappa|v|^2}F(\lambda,\mu;x)\,,\\\label{tp2}
F(\lambda+\kappa v,\mu; x)=F(\lambda,\mu+\kappa v;x)= e^{2\pi i\langle x+\rho,v\rangle}F(\lambda,\mu;x)\,.
\end{gather}
Below we mostly write $F(x)$ for $F(\lambda,\mu;x)$.

The sum $\langle F(x)\rangle_\xi=\sum_{x\in\xi+P} F(x)$ is well-defined if $\xi$ belongs to the following set:
\[
\vreg=\{\xi\in\v\,|\,w(\xi+\gamma)\ne 0\quad\forall\ \gamma\in P\}\,.
\]
The complement $\v\setminus\vreg$ is a union of hyperplanes, each given locally by $q_\alpha^{s}q^{\langle\alpha,x\rangle}=1$ for some $\alpha\in R$ and $s\in\Z$. This set of hyperplanes is locally finite and $P$-invariant, thus for every $\xi\in\vreg$ there exist a constant $\epsilon=\epsilon(\xi)>0$ such that $|w(x)|>\epsilon$ for all $x\in \xi+P$. For such $\xi$ the sum $\sum_{x\in\xi+P} F(x)$ is absolutely convergent, due to the exponentially-quadratic factor $q^{|x|^2/2}$ and the fact that $1/w(x)$ remains bounded. Therefore, $f(\xi):=\langle F(x)\rangle_\xi$ is holomorphic on $\vreg$. We claim that $f(\xi)$ extends to an entire function on $\v$.

To see that, let us look at the behaviour of $f(\xi)$ near the hypersurface
\[
\pi_{\alpha,s}:=\{\xi\in\v\,|\,q_\alpha^{s}q^{\langle\alpha,\xi\rangle}=1\}\,.
\]
We have
\begin{equation}\label{sumg}
f(\xi)(1-q_\alpha^{s}q^{\langle\alpha,\xi\rangle})=
\sum_{x\in \xi+ P}\psi(\lambda, x)\psi(-\mu, x)q^{|x|^2/2}
\frac{1-q_\alpha^{s}q^{\langle\alpha,\xi\rangle}}{w(x)}\,.
\end{equation}
Choose $\xi_0\in\pi_{\alpha,s}$ away from the hyperplanes $\pi_{\beta,r}$ with $\beta\ne \alpha$. Then there exist a  constant $C$ such that for all $\xi$ near $\xi_0$
\[
\left|\frac{1-q_\alpha^{s}q^{\langle\alpha,\xi\rangle}}{w(x)}
\right|<C\quad\text{for all}\ x\in\xi+P\,.
\]
As a result, the sum \eqref{sumg} converges absolutely and uniformly for all $\xi$ near $\xi_0$. This implies that $f(\xi)$ has at most first order pole along $\pi_{\alpha,s}$, and its residue is the (absolutely convergent) sum of the residues of the terms $F(\xi+\gamma)$. In every subsum $\sum_{r\in\Z}F(\xi+\gamma_0+r\alpha)$ there are exactly $2m_\alpha$ terms with a pole along $\pi_{\alpha,s}$, and their residues sum to zero due to Lemma \ref{cancel} and \eqref{gain}. As a result, $f(\xi)$ has a removable pole along $\pi_{\alpha,s}$, as needed.

Having established analyticity of $f(\xi)=\langle F(x)\rangle_\xi$, we now look at its translation properties.
It is clearly periodic with respect to $P$. It follows from \eqref{tp1} that for $v\in Q^\vee$
\[
f(\xi+\kappa v)=f(\xi)\,e^{2\pi i\langle \xi+\lambda-\mu,v\rangle}e^{\pi i\kappa|v|^2}\,.
\]
Now a simple check shows that the function $\langle q^{\frac{|x+\lambda-\mu|^2}{2}}\rangle_\xi$ has the same translation properties in $\xi$-variable. A standard simple fact from the theory of theta-functions tells us that these two functions must differ by some factor independent of $\xi$. We record this in the following form:
\begin{equation}\label{alm}
\langle F(x)\rangle_{\xi}=\varphi(\lambda,\mu)q^{-|\lambda-\mu|^2/2}\langle q^{\frac{|x+\lambda-\mu|^2}{2}}\rangle_\xi\,,
\end{equation}
for some entire function $\varphi(\lambda,\mu)$. It remains to relate $\varphi$ to $\psi(\lambda,\mu)$.

Using \eqref{tp2} and \eqref{alm}, it is easy to see that
\[
\varphi(\lambda+\kappa v,\mu)=\varphi(\lambda,\mu+\kappa v)=e^{2\pi
i\langle{\rho,v\rangle}}\varphi(\lambda,\mu) \quad\forall\ v\in
Q^\vee\,.
\]
As a result, $\varphi$ can be presented as a convergent (Fourier) series of the following form:
\begin{equation}\label{fou}
\varphi(\lambda,\mu)=q^{\langle\lambda+\mu,\rho\rangle}\sum_{\nu,\nu'\in
P}a_{\nu\nu'}q^{\langle\lambda,\nu\rangle}q^{\langle\mu,\nu'\rangle}\,.
\end{equation}
We want to show that this series is finite. For that we will look at the asymptotics of $\varphi$ as $\lambda,\mu\to\infty$. To get the asymptotics for $\varphi(\lambda,\mu)$, we check the behaviour of the left-hand side in \eqref{alm}.

Switching $x,\lambda$ in \eqref{psi} and \eqref{expand}, we present $\psi$ as a finite sum of the form
\begin{equation}\label{as}
\psi(\lambda,x)=q^{\langle\lambda, x+\rho\rangle}\sum_{\nu\in
P_-}\Gamma_\nu(x)q^{\langle\nu,\lambda\rangle}\,,
\end{equation}
with $\Gamma_0=\psi_\rho=\Delta(x)$ and $\Gamma_\nu=\psi_{\nu+\rho}(\lambda)$. Since the support of $\psi(\lambda,x)$ in the $x$-variable is $\lambda+\nn$, we have that $\mathrm{supp}\,\Gamma_\nu\subseteq \nn$ for all $\nu$.

Let $D_\eta=i\eta+\V$ for some generic $\eta\in\V$. Then the same arguments as in \cite[Lemma~8.1]{EV} prove the following result.
\begin{lemma}
For all $\nu$, $\Gamma_\nu/\Gamma_0$ is bounded from above when restricted to $D_\eta$. \hfill$\Box$
\end{lemma}
This lemma and \eqref{as} have the following consequence.
\begin{cor} Let $c(\lambda)=\max_{\alpha\in R_+}\langle\alpha,\lambda\rangle$. We have uniformly for all $x\in D_\eta$:  $$\psi(\lambda,x)=q^{\langle\lambda, x+\rho\rangle}\Delta(x)(1+O(q^{-c(\lambda)}))\quad\text{as}\ c(\lambda)\to-\infty\,.$$\hfill$\Box$
\end{cor}
Using this result we obtain a uniform asymptotics for the function \eqref{f} on $D_\eta$:
\[
F(x)=q^{\langle\lambda-\mu, x\rangle}q^{\langle \lambda+\mu,\rho\rangle}q^{|x|^2/2}(1+O(q^{-c}))
\]
as $c:=\max\{c(\lambda), c(\mu)\}$ tends to $-\infty$.

It follows that for $\xi\in D_\eta$
\[
\langle F(x)\rangle_{\xi}=q^{\langle \lambda+\mu,\rho\rangle}\langle q^{\langle\lambda-\mu, x\rangle}q^{|x|^2/2}\rangle_\xi(1+O(q^{-c}))\,.
\]
Substituting this in \eqref{alm} and assuming $\lambda,\mu\in P$, we conclude that
\[
q^{-\langle \lambda+\mu,\rho\rangle}\varphi(\lambda,\mu)=1+O(q^{-c})\,.
\]
Since $\lambda, \mu$ tend to infinity independently, this implies that
\[
\varphi(\lambda,\mu)=q^{\langle
\lambda+\mu,\rho\rangle}\sum_{\nu,\nu'\in
P_-}a_{\nu\nu'}q^{\langle\lambda,\nu\rangle}
q^{\langle\mu,\nu'\rangle})\,,\qquad a_{00}=1\,.
\]
Taking into account asymptotics in various Weyl chambers, we obtain that
\begin{equation*}
\varphi(\lambda,\mu)=\sum_{\nu,\nu'\in
\nn\cap\rho+P}\varphi_{\nu\nu'}q^{\langle\lambda,\nu\rangle}q^{\langle\mu,\nu'\rangle}\,,
\end{equation*}
with $\varphi_{\rho\rho}=1$. Therefore, the function $\wpsi(\lambda,\mu):=q^{\langle \lambda, \mu\rangle}\varphi(\lambda,\mu)$
will have the form as in \eqref{psi1}.

Note that by \eqref{alm} we have
\[
\langle F(x)\rangle_{\xi}=\wpsi(\lambda,\mu)q^{-(|\lambda|^2+|\mu|^2)/2}\langle q^{\frac{|x+\lambda-\mu|^2}{2}}\rangle_\xi\,.
\]
The left hand-side obviously inherits from $\psi$ the properties \eqref{axpsi} in $\lambda$, $\mu$. Also, the expression $\langle q^{\frac{|x+\lambda-\mu|^2}{2}}\rangle_\xi$ in the right-hand side is $P$-periodic in $\lambda$, $\mu$, so it satisfies \eqref{axpsi} trivially. As a result, $\wpsi(\lambda,\mu)$ must have properties \eqref{axpsi} as well. Note that, by construction, we have $\psi(\lambda,\mu)=\psi(\mu,\lambda)$.

We see that $\wpsi$ has the same properties as the normalized BA function $\psi$, therefore they differ by a constant factor. The normalized $\psi$ has $\psi_{\rho\rho}=C^{-1/2}$, while $\wpsi_{\rho\rho}=1$. Thus, $\wpsi=C^{1/2}\psi$. This finishes the proof of the theorem.
\end{proof}

We can use the generalized Weyl formula \eqref{weyl2} to obtain a symmetric version of the above theorem, thus recovering Cherednik's result \cite[Theorem 1.3]{Ch3}. We will use the notation of Theorem \ref{mm}.

\begin{theorem}[cf. {\cite[Theorem 1.3]{Ch3}}] Let $\nabla(x)=\nabla(x;q,q^{m+1})$ and $p_\lambda(x)=p_\lambda(x;q,q^{m+1})$ in case {\bf b}. Then for any $\lambda,\mu \in P_{++}$ and any $\xi\in\v$ we have
\begin{multline*}
\sum_{x\in\xi+P} p_\lambda(x)p_\mu(-x)q^{|x|^2/2}\nabla(x)\\=(-1)^{\wmm}C^{1/2}|W|q^{-\frac{|\wla|^2+|\wmu|^2}{2}}\wdel(\wmu)p_\lambda(\wmu)
q^{|\xi|^2/2}\theta(\xi)\,,
\end{multline*}
where $\theta(x)$ is the theta function \eqref{theta}, $\wla=\lambda+\wrho$, $\wmu=\mu+\wrho$ and $\wmm=\sum_{\alpha\in R_+}(m_\alpha+1)$.
\end{theorem}
This is checked in the same way as Theorem \ref{mm}. \hfill$\Box$

\section{Gaussian integrals, twisted BA functions and twisted Macdonald--Ruijsenaars model}\label{twgauss}

Let us consider for a moment what happens if we replace the Gaussian $q^{-|x|^2/2}$ in Theorem \ref{mehta} by $q^{-a|x|^2/2}$ with  $a>0$. For Proposition \ref{res1} and the cancelation of residues to work, we need the function $g(x)=q^{-a|x|^2/2}$ to be quasi-invariant, i.e. take equal values along the shifted hyperplanes:
\begin{equation}\label{gres}
    g(x-\frac12j\alpha)=g(x+\frac12j\alpha)\quad\text{for}\ q^{\langle \alpha, x\rangle}=1\quad\text{and}\ j\in\Z\,.
\end{equation}
We have
\begin{equation*}
g(x-\frac12j\alpha)/g(x+\frac12j\alpha)=q^{aj\langle\alpha,x\rangle}\,.
\end{equation*}
Therefore, \eqref{gres} will hold as soon as $a\in\N$.

So, let us take $a=\ell\in\N$ and consider the integral
\begin{equation}\label{il}
\int_{C_\xi}\frac{\psi(\lambda,x)\psi(\mu, x)}{\Delta(x)\Delta(-x)}q^{-\ell|x|^2/2}\,dx\,.
\end{equation}
It turns out that this integral is still `computable', but for $\ell>1$ the result will be expressed in terms of a new function $\psil$ whose properties are similar to those of $\psi$. This `twisted' BA function $\psil$ will be a common eigenfunction for a certain quantum integrable model given by commuting $W$-invariant difference operators that generalize the Macdonald operators $D^\pi$. To the best of our knowledge, this model is new; in the case $R=A_n$ it generalizes the trigonometric Ruijsenaars model \cite{R}. This will be explained in Sections \ref{twba}--\ref{twsum} below. Finally, Section \ref{ac} discusses the situation in cases {\bf a} and {\bf c}. Note that in those cases theorem \ref{mehta} is not true. Indeed, in case {\bf a}, for instance, the function $g(x)=q^{-|x|^2/2}$ does not satisfy the relevant quasi-invariance properties:
\begin{equation*}
g(x-\frac12 j\alpha')/g(x+\frac12j\alpha')=q^{j\langle\alpha',x\rangle}\ne q^{j\langle\alpha,x\rangle}
\end{equation*}
as soon as $\alpha'\ne\alpha$. As a result, the cancelation of residues will not work unless we replace $q^{-|x|^2/2}$ by $g(x)=q^{-\ell|x|^2/2}$ with suitably chosen $\ell$. Thus, while Theorem \ref{mehta} is not true in cases {\bf a} and {\bf c}, its twisted version with suitable $\ell$ will work. This is discussed in Sec.~\ref{ac} below.

\subsection{Twisted BA functions}\label{twba}

We keep the notation of Section \ref{baf}. In this section we consider case {\bf b} only, so all the notation of Section \ref{baf} applies with $(\rst,\mst)=(R,m)$ and $q_\alpha=q^{|\alpha|^2/2}$.

For a reduced irreducible root system $R$, a $W$-invariant set of labels $m_\alpha\in\Z_+$, and an integer $\ell\in\N$, a twisted BA function $\psil$ (of type {\bf b}) has the following form:
\begin{equation}\label{psil}
\psil(\lambda,x) = q^{\langle\lambda,
x\rangle/\ell}\sum_{\nu\in\nn\cap\,\rho+\ell^{-1}P}
\psi_\nu(\lambda) q^{\langle\nu,x\rangle}\,,
\end{equation}
where $\nn$ is the polytope \eqref{nu}.

The function $\psil$ must also satisfy further conditions, similar to \eqref{axpsi}.
Namely, we require that for
each $\alpha\in R$, $j=1,\dots ,m_\alpha$ and any $\epsilon$ with $\epsilon^\ell=1$ we have
\begin{equation}\label{axpsil}
\psil\left(\lambda,\, x-{\frac 12 j\alpha}\right) = \epsilon^{j}\psil\left(\lambda,
\,x{+\frac 12 j\alpha}\right)\quad\text{for}\ \
q^{\langle\alpha,x\rangle/\ell}=\epsilon\,.
\end{equation}

\begin{definition}
A function $\psil(\lambda, x)$ with the properties
\eqref{psil}--\eqref{axpsil} is called a twisted Baker--Akhiezer function associated to the data $\{R,m,\ell\}$.
\end{definition}

For $\ell=1$ this is the definition of Section \ref{baf}. Our goal is to prove the following two results.

\begin{theorem}\label{baul} (1) A twisted Baker--Akhiezer function $\psil(\lambda, x)$ exists and is unique up to multiplication by a factor depending on $\lambda$.

(2) Let us normalize $\psil$ by requiring \eqref{norm0} (recall that $\Delta'=\Delta$ in case {\bf b}). Then we have $$\psil(\lambda,x)=\psil(x,\lambda)\,.$$

(3) As a function of $x$, $\psil$ is a common eigenfunction of certain pairwise commuting $W$-invariant difference operators $D^{\pi}_\ell$, $\pi\in P_{++}$, namely,
\begin{equation*}
    D^{\pi}_\ell \psil = \mathfrak m_\pi(\lambda)\psil\,,\qquad  \mathfrak m_\pi(\lambda)=\sum_{\tau\in W\pi}
q^{\langle\tau,\lambda\rangle}\,.
\end{equation*}
The operators $D^{\pi}_\ell$ have the same leading terms as $(D^\pi)^\ell$, i.e. they are lower-term perturbations of the Macdonald operators raised to the $\ell$th power.
\end{theorem}

\begin{theorem}\label{mehtal}  For any $\ell\in\N$, any $\lambda,\mu\in\v$ and big $\xi\in\V$ we have
\begin{equation}\label{mehtaintl}
\int_{C_\xi}\frac{\psi(\lambda,x)\psi(\mu, x)}{\Delta(x)\Delta(-x)}q^{-\ell|x|^2/2}\,dx=(-1)^MC^{-1/2}
q^{\frac{|\lambda|^2+|\mu|^2}{2\ell}} \psil(\lambda,\mu)\,,
\end{equation}
where $\psil$ is the normalized twisted BA function and $C$, $M$ are the same as in Theorem \ref{mehta}.
\end{theorem}

Theorem \ref{baul} is analogous to Theorem \ref{bau}, but we cannot use the same method to prove it. The reason is that the arguments of \cite{C02} exploit in an essential way Macdonald operators and their properties. In the twisted case there exist certain analogues of these operators (these are $D^\pi_\ell$ appearing in Theorem \ref{baul}), but we cannot write them down explicitly. In fact, the existence of $D^\pi_\ell$ will be established only once we know the existence of $\psil$. So we change our tack: we will instead {\it define} $\psil$ by the formula \eqref{mehtaintl} and from that we will derive the required properties \eqref{psil}--\eqref{axpsil}.

\subsection{Proof of Theorem \ref{baul}} Let $I_\ell(\xi)$ denote the integral \eqref{il}. Since the function $g(x)=q^{-\ell|x|^2/2}$ satisfies \eqref{gres}, the residues of the integrand cancel as in Lemma \ref{cancel}, therefore, $I_\ell(\xi)$ will not depend on $\xi$ provided that it is big.

Now we compute the integral using series expansion as in Section \ref{expansion}.
Assuming $\xi$ belongs to the negative Weyl chamber, we get similarly to \eqref{expint} that
\begin{multline}\label{expintl}
I_\ell(\xi)=\sum_{\gamma\in P_-}f_{\gamma}\int_{C_\xi} q^{\langle
\lambda+\mu+\gamma, x\rangle}q^{-\ell|x|^2/2}\,dx =\\
\sum_{\gamma\in P_-}f_\gamma
q^{\frac{1}{2\ell}|\lambda+\mu+\gamma|^2}=q^{\frac{1}{2\ell}|\lambda+\mu|^2}\sum_{\gamma\in
P_-}f_\gamma q^{\frac{1}{\ell}\langle \gamma,
\lambda+\mu\rangle}q^{\frac{1}{2\ell}|\gamma|^2}\,,
\end{multline}
where the coefficients $f_\gamma(\lambda,\mu)$ are exactly the same as in \eqref{expint}.

Comparing such expansions for different chambers, we conclude (in
the same way as in Section \ref{expansion}) that $I_\ell$ is an
elementary function of the form
\begin{equation*}
    I_\ell= q^{\frac{1}{2\ell}|\lambda+\mu|^2}\sum_{\nu\in\nn\cap\,\rho+{\ell}^{-1}P} \psi_\nu(\lambda) q^{\langle\nu,\mu\rangle}\,.
\end{equation*}
Therefore, $I_\ell$ has the form $q^{\frac{1}{2\ell}(|\lambda|^2+|\mu|^2)}\psil(\lambda,\mu)$ where $\psil$ has the required form \eqref{psil} (with $x$ replaced by $\mu$).

As a function of $\mu$, $I_\ell$ has the properties \eqref{axpsi}. As a result, we obtain that for $j=1,\dots,m_\alpha$ and for $q^{\langle\alpha,\mu\rangle}=1$
\[
q^{\frac{1}{2\ell}|\mu-\frac12 j\alpha|^2}\psil(\lambda, \mu-\frac12 j\alpha)=q^{\frac{1}{2\ell}|\mu+\frac12 j\alpha|^2}\psil(\lambda, \mu+\frac12 j\alpha)\,.
\]
It is easy to see that these are equivalent to conditions \eqref{axpsil} (again, with $x$ replaced by $\mu$).

This proves the existence of a function $\psil$ satisfying \eqref{psil}--\eqref{axpsil}. To show that it is nonzero, we compute the leading coefficient $\psi_{\rho}$. This is done by using the formula \eqref{gamma1}, which still applies in our case. The result is that $\psi_{\rho}=(-1)^MC^{-1/2}\,\Delta(\lambda)$. The uniqueness of $\psil$, up to a factor depending on $\lambda$, can be proved by exactly the same arguments as in \cite[Proposition 3.1]{C02}.
Finally, $I_\ell$ is obviously symmetric in $\lambda$ and $\mu$, therefore, $\psil(\lambda,\mu)$ is also symmetric:
\[
\psil(\lambda,\mu)=\psil(\mu,\lambda)\,.
\]
Thus, the Gaussian integral automatically gives us the normalized and symmetric $\psil$. This finishes the proof of parts (1) and (2) of Theorem \ref{baul}.

Part (3) follows from the uniqueness of $\psil$ by the so-called
Krichever's argument familiar in the finite-gap
theory\cite{Kr1,Kr2}, cf. \cite[Section 5.1]{C02}. Namely, recall
the ring $\mathcal Q\subset \A$ of quasi-invariants \eqref{quasi}.
Note that $\A^W\subset \mathcal Q$. We have the following result.

\begin{theorem}[cf. {\cite[Theorem 5.1 and Proposition 5.3]{C02}}]
(1) For each $f(x)\in\mathcal Q$ there exists a difference operator $D_f$ in $\lambda$-variable on the lattice $P$ such that $\psil(\lambda,x)$ is its eigenfunction: $D_f\psil=f(x)\psil$. All these operators pairwise commute.

(2) For any dominant weight $\pi\in P_{++}$ and $f=\mathfrak m_\pi(x)$ the corresponding operator $D_f$ is $W$-invariant and has the form
\[
D_f=\sum_{\tau\in W\pi} a_\tau T^{\ell\tau} + l.o.t.\,,
\]
where the leading coefficients $a_\tau$ are given by
\[
a_\pi(\lambda)=\frac{\Delta(\lambda)}{\Delta(\lambda+\ell\pi)}\,,\quad a_{w\pi}(\lambda)=a_\pi(w^{-1}\lambda)\,.
\]
\end{theorem}

This is proved in the same way as \cite[Theorem 5.1 and Proposition 5.3]{C02}.

Part (3) of Theorem \ref{baul} now follows immediately from this result after switching between $x$ and $\lambda$.
This completes the proof of Theorem \ref{baul}.

\subsection{Proof of Theorem \ref{mehtal}}

This is now immediate from above: we have
$$I_\ell(\xi)=(-1)^MC^{-1/2}q^{\frac{1}{2\ell}(|\lambda|^2+|\mu|^2)} \psil(\lambda,\mu)\,,$$
where $\psil$ will satisfy all the properties of the normalized twisted BA function.
\hfill$\Box$

\subsection{Summation formula}\label{twsum} We can also generalize the summation formula to the twisted case. Since the arguments are entirely analogous, we only formulate the result.

\begin{theorem}\label{suml} For any $\lambda, \mu\in \V$ and $\xi\in\v$ we have
\begin{equation*}\label{gsuml}
\sum_{x\in \xi+P} \frac{\psi(\lambda,x)\psi(\mu,-x)}{\Delta(x)\Delta(-x)}q^{\ell|x|^2/2}=
C^{1/2}q^{-\frac{|\lambda|^2+|\mu|^2}{2\ell}}\psil(\lambda,\mu)\sum_{x\in\xi+P} q^{\frac{\ell}{2}|x+\frac{\lambda-\mu}{\ell}|^2}\,,
\end{equation*}
where $C$ is the constant \eqref{cc} and $\psil$ is the twisted BA function associated to $(R,m,\ell)$. 

\end{theorem}

One can use Theorems \ref{mehtal}, \ref{suml} and the generalized Weyl formula \eqref{weyl2} to express, in terms of the twisted BA functions, the integrals and sums
\[
\int_{i\V}p_\lambda(x)p_\mu(x)\nabla(x)q^{-\ell|x|^2/2}\,dx\,,\qquad \sum_{x\in\xi+P}p_\lambda(x)p_\mu(x)\nabla(x)q^{\ell|x|^2/2}
\]
for $p_\lambda=p_\lambda(x;q,q^{m+1})$, $\nabla=\nabla(x;q,q^{m+1})$ in case {\bf b}. In particular, this gives an expression for $\int_{i\V}\nabla(x)q^{-\ell|x|^2/2}\,dx$. In general, however, this does not seem to lead to a nice factorized form as in the case $\ell=1$.

\subsection{Twisted BA functions in cases {\bf a} and {\bf c}}\label{ac}

Let us consider the Gaussian integrals for the remaining cases {\bf a} and {\bf c} of Macdonald's theory. Note that case {\bf a} for $R=A,D,E$ is the same as case {\bf b} if we choose the scalar product so that all roots have length $\sqrt 2$. Thus, the only cases not covered by Theorems \ref{mehta} and \ref{mehtal} are case {\bf c} (when $R=C_n$) and case {\bf a} for $R=B_n, C_n, F_4, F_4^\vee, G_2, G_2^\vee$. Also note that the cases $R=F_4$ and $R=F_4^\vee$ are equivalent because these roots systems are isomorphic, and the same is true for $G_2$, while the cases $R=B_n$ and $R=C_n$ can be obtained from case {\bf c} by a suitable specialization of the parameters $m_i$.

Let $\psi=\psi_{R,m}$ be the corresponding normalized BA function. Consider the integral \eqref{il}.
As a starting point, we would like that integral to be independent of $\xi$ (provided that it is big). To have the cancelation of residues as in Lemma \ref{cancel}, we need $g(x)=q^{-\ell|x|^2/2}$ to satisfy the properties
\[
g(x-\frac12j\alpha')=g(x+\frac12j\alpha')\quad\text{for}\ q^{\langle\alpha,x\rangle}=1\,,
\]
where $\alpha\in R$ in case {\bf a} or $\alpha\in R^2$ in case {\bf c}. In addition to that, in case ${\bf c}$ we need that
\begin{equation}\label{adc}
g(x-se_i)=g(x+se_i)\quad\text{for}\ q^{x_i}=\pm 1\,,
\end{equation}
where $s\in\frac12\Z$.

This puts the following restrictions on $\ell$ in case {\bf a}:
\begin{equation}\label{ela}
 \ell\in \frac12|\alpha|^2\Z\quad\text{for all}\ \alpha\in R\,.
\end{equation}
If we assume that the {\it short} roots in $R$ have length $\sqrt2$, then we have
\begin{equation*}
\ell\in\begin{cases}\,\Z\quad R=A_n, D_n, E_{6-8}\,,\\\,2\Z\quad R=B_n, C_n, F_4\,,\\\,3\Z\quad R=G_2\,.
\end{cases}
\end{equation*}
In general, let $\nu_R$ denote
\[
\nu_R=\max_{\alpha\in R}\{|\alpha|^2/2\}\,,
\]
then our conditions on $\ell$ can be written in all cases as
\begin{equation}\label{rest}
\ell\in \nu_R\Z\cap (\nu_{R^\vee})^{-1}\Z\,.
\end{equation}

In case {\bf c}, we obtain from \eqref{adc} that $\ell\in 2\Z$. In fact, since we only need \eqref{adc} to hold for certain half-integral $s$ (see \eqref{axpsi1}--\eqref{axpsi2}), it is possible to choose $\ell\in \Z$ if either $m_3$ or $m_4$ is $1/2$. We will ignore this option, and will always assume for simplicity that
$\ell\in 2\Z$ in case {\bf c}.

Our goal is to show that the integral \eqref{il} for such $\ell$ can be expressed in terms of a suitably defined BA function. What looks particularly peculiar in cases {\bf a }, {\bf c} is that the usual BA function $\psi(\lambda,x)$ is not self-dual, $\psi(\lambda,x)\ne \psi(x,\lambda)$, since one has also to switch from $(R,m)$ to $(R',m')$ under the duality. However, the twisted BA functions $\psil$ defined below are always self-dual, even for the case {\bf c} with full five parameters $m_1,\dots, m_5$.

\medskip

So, let $(R,m)$ be of type {\bf a} or {\bf c}, in the notation of Section \ref{nota}. That is, in case {\bf a} we consider a reduced root system $R$ and $W$-invariant integers $m_\alpha\in \Z_+$, and put $(\rst,\mst)=(R^\vee,m)$. In case {\bf c}, we take $R=R'=C_n$ with $m,\mst$ being (half-)integers $m_i$ and $\mst_i$ \eqref{dualc}.

Choose $\ell$ such that
\begin{equation}\label{restd}
\ell\in \nu_{R'}\Z\cap (\nu_{R'^\vee})^{-1}\Z\,.
\end{equation}
(This is the choice, dual to \eqref{rest}. In case {\bf c} this still means $\ell\in 2\Z$.)

A twisted BA function $\psil$ in cases {\bf a} or {\bf c} has the same form \eqref{psil}:
\begin{equation*}
\psil(\lambda,x) = q^{\langle\lambda,
x\rangle/\ell}\sum_{\nu\in\nn\cap\,\rho+ \ell^{-1}P}
\psi_\nu(\lambda) q^{\langle\nu,x\rangle}\,.
\end{equation*}
It must also satisfy further conditions, similar to \eqref{axpsil}.
Namely, for each $\alpha\in R$ (in case {\bf a}) or $\alpha\in R^2$ (in case {\bf c}), any $j=1,\dots ,m_\alpha$ and any $\epsilon$ with $\epsilon^\ell=1$ we have
\begin{equation}\label{axpsia}
\psil\left(\lambda,\, x-{\frac 12 j\alpha'}\right) = \epsilon^{j}\psil\left(\lambda,
\,x{+\frac 12 j\alpha'}\right)\quad\text{for}\ \
q^{\langle\alpha,x\rangle/\ell}=\epsilon\,.
\end{equation}
In case {\bf c}, we require additionally for each $\alpha=e_i\in R^1$ the following: for any $\epsilon$ with $\epsilon^{2\ell}=1$

(1) for all $0<s\preccurlyeq(m_1,m_2)$
\begin{equation}\label{axpsia1}
\psi(\lambda, x-se_i)=\epsilon^{2s}\psi(\lambda,x+se_i)\quad\text{ for}\ q^{x_i/\ell}=\epsilon\,,\quad\text{provided}\ \epsilon^\ell=1;
\end{equation}

(2) for all $0<s\preccurlyeq(m_3,m_4)$
\begin{equation}\label{axpsia2}
\psi(\lambda, x-se_i)=\epsilon^{2s}\psi(\lambda,x+se_i)\quad\text{ for}\ q^{x_i/\ell}=\epsilon\,,\quad\text{provided}\ \epsilon^\ell=-1\,.
\end{equation}

\begin{definition} Let $\ell$ be as in \eqref{restd}.
A function $\psil(\lambda, x)$ of the form
\eqref{psil} satisfying conditions \eqref{axpsia}--\eqref{axpsia2} is called a twisted Baker--Akhiezer function of type {\bf a} or {\bf c}, respectively, associated to the data $\{R,m,\ell\}$.
\end{definition}

Now the same arguments as in case {\bf b} prove the following results.

\begin{theorem}\label{baua} (1) A twisted Baker--Akhiezer function $\psil(\lambda, x)$ exists and is unique up to multiplication by a factor depending on $\lambda$.

(2) Let us normalize $\psil$ by requiring \eqref{norm0}. Then we have $$\psil(\lambda,x)=\psil(x,\lambda)\,.$$

(3) As a function of $x$, $\psil$ is a common eigenfunction of certain pairwise commuting $W$-invariant difference operators $D^{\pi}_\ell$, $\pi\in P_{++}$, namely,
\begin{equation*}
    D^{\pi}_\ell \psil = \mathfrak m_\pi(\lambda)\psil\,,\qquad  \mathfrak m_\pi(\lambda)=\sum_{\tau\in W\pi}
q^{\langle\tau,\lambda\rangle}\,.
\end{equation*}
The operators $D^{\pi}_\ell$ have the same leading terms as $(D^\pi)^\ell$, i.e. they are lower-term perturbations of the Macdonald operators raised to the $\ell$th power.
\end{theorem}

\begin{theorem}\label{mehtaa}  Let $\psi(\lambda,x)$ be the normalized BA function associated to $(R,m)$ in cases {\bf a} or {\bf c}. Let $\ell$ be as in \eqref{rest}. For any $\lambda,\mu\in\v$ and big $\xi\in\V$ we have
\begin{equation}\label{mehtainta}
\int_{C_\xi}\frac{\psi(\lambda,x)\psi(\mu, x)}{\Delta(x)\Delta(-x)}q^{-\ell|x|^2/2}\,dx=(-1)^MC^{-1/2}
q^{\frac{|\lambda|^2+|\mu|^2}{2\ell}} \psi'_{\ell}(\lambda,\mu)\,,
\end{equation}
where $C$, $M$ are the same as in Theorem \ref{mehta} and $\psi'_\ell=\psi_{\rst,\mst,\ell}$ is the normalized twisted BA function associated to the dual data $(\rst,\mst, \ell)$.
\end{theorem}

We also have the related summation formulas, similar to Theorem \ref{suml} and proved in the same way.
\begin{theorem}\label{sumla} Assume the notation of Theorem \ref{mehtaa}. For any $\lambda, \mu\in \V$ and $\xi\in\v$ we have
\begin{equation*}
\sum_{x\in \xi+P'} \frac{\psi(\lambda,x)\psi(\mu,-x)}{\Delta(x)\Delta(-x)}q^{\ell|x|^2/2}=
C^{1/2}q^{-\frac{|\lambda|^2+|\mu|^2}{2\ell}}\psi'_\ell(\lambda,\mu)\sum_{x\in\xi+P'} q^{\frac{\ell}{2}|x+\frac{\lambda-\mu}{\ell}|^2}\,,
\end{equation*}
where $P'=P(R')$ is the weight lattice of $R'$ and $\psi'_\ell$ is the twisted BA function of type {\bf a} or {\bf c}, associated to $(\rst,\mst,\ell)$.
\end{theorem}

\begin{remark}
In \cite{St} some analogues of Cherednik--Macdonald--Mehta identities are obtained for Koornwinder polynomials, with a suitably chosen periodic version of the Gaussian. This does not seem to be directly related to the Gaussians used above, so it is not clear to us whether our methods can be adapted to the Gaussians from \cite{St}.
\end{remark}

\end{document}